\documentclass{article}

\usepackage{amsmath,amsthm,indentfirst}
\usepackage{amssymb}
\usepackage{amsfonts}
\usepackage{amscd}
\usepackage[latin1]{inputenc}

\theoremstyle{plain}
\newtheorem*{maintheorem*}{Main Theorem}
\newtheorem*{thm*}{Theorem}
\newtheorem*{thma*}{Theorem A}
\newtheorem*{thmaa*}{Theorem A'}
\newtheorem*{thmb*}{Theorem B}
\newtheorem*{thmo*}{Theorem 1.1}
\newtheorem*{thmc*}{Theorem C}
\newtheorem*{thmd*}{Theorem D}
\newtheorem*{thmf*}{Theorem 4.1}
\newtheorem*{remark*}{Remark}
\newtheorem*{conjecture*}{Conjecture}
\newtheorem*{prop*}{Proposition}
\newtheorem{thm}{Theorem}[section]
\newtheorem{cor}[thm]{Corollary}
\newtheorem{lem}[thm]{Lemma}

\theoremstyle{definition}

\newtheorem*{proofc*}{Proof of Theorem C}

\newtheorem{definition}[thm]{Definition}

\newtheorem{remark}[thm]{Remark}

\def\bbz{\mathbb{Z}}
\def\bbq{\mathbb{Q}}
\def\bbf{\mathbb{F}}
\def\bbr{\mathbb{R}}

\def\bbn{\mathbb{N}}

\def\bfr{\mathfrak{B}}
\def\pfr{\mathfrak{P}}
\def\ybf{\mathbf{y}}
\def\bbf{\mathbf{B}}
\def\dbf{\mathbf{D}}
\def\wbf{\mathbf{w}}
\def\vare{\varepsilon}

\def\e{\mathbf{e}}

\def\d{\partial}
\def\g{\nabla}
\def\p{\prod_{\nu\in S}}
\def\tp{\prod_{\nu\in\tilde{S}}}
\def\u{\mathcal{U}}
\def\U{\mathbf{U}}
\def\x{\mathbf{x}}
\def\z{\mathbf{z}}
\def\fbf{\mathbf{f}}
\def\bff{\mathbf{F}}
\def\diag{{\rm diag}}
\def\cf{\check{f}}

\def\B{\rm{B}}

\def\GL{\rm{GL}}
\def\rank{\rm{rank}}
\def\q{\mathbf{q}}
\def\para{\|}
\def\h{\hspace{1mm}}
\def\tq{\tilde{\q}}
\def\nuball{B(x_{\nu},\frac{1}{4R\|q\g f_{\nu}(x_{\nu})\|_{\nu}})}
\def\ball{\prod_{\nu\in S}\nuball}
\def\ts{\tilde{S}}
\def\tp{\prod_{\nu\in\ts}}

\title{\sc Simultaneous Diophantine Approximation in Non-degenerate $p$-adic Manifolds.}
\author{A.~Mohammadi
\and
A.~Salehi Golsefidy}
\date{}

\begin{document}
%%%%%%%%%%%%%%%%%%%%%%%%%%%%%% Abstract %%%%%%%%%%%%%%%%%%%%%%%%%%%%%%%%%
\maketitle
\begin{abstract}

\noindent
$S$-arithmetic Khintchine-type theorem for products of non-degenerate\\ analytic $p$-adic manifolds is proved for the convergence case. In the $p$-\\adic case the divergence part is also obtained. 
\end{abstract}

%%%%%%%%%%%%%%%%%%%%%%%%%%%%%% Introduction %%%%%%%%%%%%%%%%%%%%%%%%%%%%%
\section{Introduction}

%%%%%%%%%%%%%%%%%%%%%%%NEWWWWW%%%%%%%%%%%%%%%%%%%%

\textbf{Metric Diophantine approximation}
The metric theory of Diophantine approximation studies the interplay between the precision of approximation of real or $p$-adic numbers by rationals and the measure of the approximated set within certain prescribed precision. The ``finer" the precision is the ``smaller" the approximated set is. The theory was initiated by A.~Khintchine, who in \cite{Kh24} proved an ``almost every"  vs ``almost no" dichotomy for $\bbr.$ Let us state Khintchine's result. Given a decreasing function $\psi:\bbr^+\rightarrow\bbr^+$ we define the notion of a {\it $\psi$-approximable} number as follows; A real number $\xi$ is called {\it $\psi$-approximable} if for infinitely many integers $p$ and $q$, 
one has $|q\xi-p|<\psi(|q|)$. It is called {\it very well approximable} (VWA) if it is $\psi_{\vare}$-A where $\psi_{\vare}=1/q^{1+\vare}$ for some positive $\vare$. A.~Khintchine showed Lebesgue almost every (resp. almost no) real number is $\psi$-A if $\sum_{q=1}^{\infty}\psi(q)$ diverges 
(resp. converges). 
%Let us be more precise. Let $\psi$ be a 
%decreasing function from $\bbr^+$ to $\bbr^+$ e.g. $\psi_{\vare}(q)=1/q^{1+\vare}$. A real number 
%$\xi$ is called {\it $\psi$-approximable} if for infinitely many integers $p$ and $q$, 
%one has $|q\xi-p|<\psi(|q|)$. It is called {\it very well approximable} (VWA) if it is $\psi_{\vare}$-A for 
%some positive $\vare$. A.~Khintchine~\cite{Kh24} has shown that almost all (resp. almost no) points, in terms of the 
%Lebesgue measure, are $\psi$-A if $\sum_{q=1}^{\infty}\psi(q)$ diverges 
%(resp. converges). 
We refer to ~\cite[Chapter IV, Section 5]{Sch}, ~\cite[Chap. 1]{BD} and~\cite[Chap. 5]{Ca} for an account on this and further historical remarks.

The metric theory of Diophantine approximation on manifolds was considered as early as 1932 when K.~Mahler~\cite{Ma} conjectured that almost no point of the Veronese curve, $(x,x^2,\cdots,x^n),$ is VWA. This conjecture drew considerable amount of attention and was finally settled affirmatively by V.~G.~Sprind\v{z}uk see in~\cite{sp1,sp2}. Sprind\v{z}uk's idea (the so called essential and non-essential domains) has been applied by many people to attack many problems stated by him in both the real and the $p$-adic setting (the definition of $\psi$-A in the $p$-adic setting is given bellow). One should mention several works conducted on this issue by V.~Beresnevich, V.~Bernik, M.~M.~Dodson, E.~Kovalevskaya and others. See for example~\cite{Ber2, Ber3, BBK, BK1, Ko2}.

In 1985 S.~G.~Dani observed a nice relationship between flows on homogenous spaces and Diophantine approximation. This point of view was taken on and pushed much further in later works of D.~Kleinbock and G.~A.~Margulis. In~\cite{KM} Kleinbock and Margulis introduced a beautiful dynamical approach to the metric theory of Diophantine approximation and settled a multiplicative version of Sprind\v{z}uk 's conjecture, we refer to their paper for the formulation and further comments. The ``almost every" vs ``almost no" dichotomy was also completed within a few years from then, see in~\cite{BKM, BBKM}. It is worth mentioning that philosophically speaking the dynamical approach in~\cite{KM} and the idea of essential and non-essential domains of  Sprind\v{z}uk are both based on a delicate covering argument.   

%To whom the theory of diophantine approximated on manifolds owes a fare amount. He formulated several conjectures in this direction. The break through however came much later when D.~Kleinbock and G.~Margulis in~\cite{KM} settled his conjecture in the general form. The dichotomy also was completed within a few years from then we refer to~\cite{BKM, Ber2, Ber3, BBKM}. The interested reader may find exact formulation and extensive history of the problem in loc. cit.   

%\textbf{$S$-arithmetic Diophantine approximation.}
%There are relatively less known results in the $p$-adic, and simultaneous approximations in different places. The most general result for $p$-adic Diophantine approximations is a recent work of  V.~Beresnevich, V.~Bernik, E.~Kovalevskaya~\cite{BBK}, in which they prove both the convergence and the divergence Khintchine-type theorem for the $p$-adic  Veronese curve, i.e. $\{(x,x^2,\cdots,x^n)|x\in\bbq_p\}$. It is worth mentioning that the convergence case had been already proved by E.~Kovalevskaya~\cite{Ko1}. There are a few other results of convergence Khintchine-type  for more general curves in the $p$-adic plane or space, e.g.~\cite{BK1,Ko2}. 

One expects, from the nature of the dynamical approach, that this approach would work just as well in the $S$-arithmetic setting. This was started by D.~Kleinbock and G.~Tomanov~\cite{KT}. They defined the notion of VWA and showed an analogue of the Sprind\v{z}uk's conjecture in the $S$-arithmetic setting. This general philosophy was taken on in~\cite{MS} where we proved a Khintchine type theorem in $S$-arithmetic setting. In that paper the assumption was that the finite set of places, $S$, contains the infinite place. And we postponed the completion of the picture to this paper.

%Situation in the simultaneous Diophantine approximation is even less clear. In a recent paper~\cite{KT}, D.~Kleinbock and G.~Tomanov came up with an $S$-arithmetic version of metric Diophantine approximation. They carefully defined the notion of extremal manifolds and proved the analogous theorem, namely they showed that a non-degenerate $\bbq_S$-manifold is extremal. We refer the reader to their paper for the definitions.  As for the Khintchine-type theorems the most general result is a recent work of   V.~Bernik and E.~Kovalevskaya\cite{BK2}. They establish an inhomogeneous convergence Khintchine-type theorem for the Veronese curve with components in product of several local fields, more specifically  $\{(\x,\x^2,\dots,\x^n)|\x\in\bbc\times\bbr\times\prod_{p\in S}\bbq_p\}$. 

Let us fix some notations and conventions which are needed in order to state the main results of this paper, these notations will be used throughout the paper. We will not define the technical terms in here but rather refer the reader to the corresponding section for the precise definitions and remarks. 

We let $S$ be a finite set of places of $\bbq$ whose cardinality will be denoted by $\kappa$ throughout. We will always assume $S$ does not contain the infinite place and let $\ts=S\cup\{\infty\}.$ Define $\bbq_S=\p \bbq_{\nu}$ and correspondingly $\bbq_{\ts}=\tp \bbq_{\nu}.$

\vspace{1mm}
\noindent
We let $\Psi:\mathbb{Z}^{n+1}\setminus\{0\}\rightarrow\bbr^+$ be a map. A vector $\xi\in\bbq_S$ is said to be $\Psi$-A if 
$$\mbox{For infinitely many}\h\h\tilde{{\bf q}}=(\q,q_0)\in\bbz^n\times\bbz\h\h\mbox{one has}\h\h|\q\cdot\xi+q_0|_S^{\kappa}\le\Psi(\tilde{\q})$$ 
As one sees readily there are two possible ways of defining $\Psi$-A vectors in $\bbq_S.$ However the above mentioned definition is what will be used in this paper. We refer to the references provided before for the ``dual" definition and related remarks. Let us just note that the notion of a VWA number which corresponds to $\Psi_{\vare}(\tilde{{\bf q}})=\| \tilde{{\bf q}}\|_S^{-(1+n)(1+\vare)}$ is unchanged if one uses either definition.

%\noindent  
%\textbf{A few terminologies and the statement of the main result.}  Here we introduce necessary notations to state the main results of the article. We do not define any technical term and refer the reader to the second section for the definitions. Let $S$ be a finite set of places of $\bbq$ which does not contain the infinite place, $\ts=S\cup\{\infty\}$, $\kappa$ the cardinality of $S$, and $\bbq_S=\p \bbq_{\nu}$. Let $\Psi$ be a map from $\mathbb{Z}^{n+1}\setminus\{0\}$ to $\bbr^+$. A vector $\xi\in \bbq_S^n$ is called $\Psi$-A if for infinitely many $\tilde{{\bf q}}=(\q,q_0)\in\bbz^n\times\bbz$ one has $|\q\cdot\xi+q_0|_S^{\kappa}\le\Psi(\tilde{\q})$.  It is worth mentioning that according to D.~Kleinbock and G.~Tomanov a vector in $\bbq_S$ is {\it very well approximable} if it is $\Psi_{\vare}$-A, where $\Psi_{\vare}(\tilde{{\bf q}})=\| \tilde{{\bf q}}\|_S^{-(1+n)(1+\vare)}$. 
For all $\nu\in S$ we fix once and for all an open bounded ball $U_{\nu}$ in $\bbq_{\nu}^{d_{\nu}}.$ Let
$$f_{\nu}=(f_{\nu}^{(1)},\cdots,f_{\nu}^{(n)}): U_{\nu} \rightarrow\bbq_{\nu}^n$$
be analytic non-degenerate map i.e $f_{\nu}$'s are analytic and the restrictions of $1,\hspace{.5mm} f_{\nu}^{(1)},\cdots,f_{\nu}^{(n)}$ to any open ball of $U_{\nu}$ are linearly independent over $\mathbb{Q}_{\nu}$.

%Let $U_{\nu}$ be an open bounded ball in
%$\bbq_{\nu}^{d_{\nu}}$ for any $\nu\in S$. Let
%$f_{\nu}=(f_{\nu}^{(1)},\cdots,f_{\nu}^{(n)}): U_{\nu} \rightarrow
%\bbq_{\nu}^n$ be an analytic map for any $\nu\in S,$ and the restrictions of $1,\hspace{.5mm} f_{\nu}^{(1)},\cdots,f_{\nu}^{(n)}$ to any open ball of $U_{\nu}$ are linearly independent over $\mathbb{Q}_{\nu}$. 
Define $\mathbf{U}=\p U_{\nu}$ and let
$\mathbf{f}(\x)=(f_{\nu}(x_{\nu}))_{\nu\in S}$, \vspace{.75mm}where $\x=(x_{\nu})_{\nu\in S}\in
\mathbf{U}$. Since $\U$ is compact we may replace $\mathbf{f}$ by
$\mathbf{f}/M$, for a suitable $\tilde S$-integer $M$, and assume
that $\para f_{\nu}(x_{\nu})\para_{\nu} \le 1,$ $\para \g
f_{\nu}(x_{\nu})\para_{\nu}\le 1$, and $L\leq1$ where
$$L=\sup
\bigcup_{|\beta|=2,\nu\in
S}\{2|\Phi_{\beta}f_{\nu}(x)|_{\nu}\hspace{1mm}|\hspace{1mm} x\in
U_{\nu}\times U_{\nu}\times U_{\nu}\},$$
for any $\nu\in S.$ The functions $\Phi_{\beta} f_{\nu}$ in here are certain two fold difference quotients of $f_{\nu},$ see section~\ref{secprelim} for the exact definition.  

We can now state the theorems.

%Let
%We have the following \textit{$S$-arithmetic Khintchine-type}  theorem

\begin{thm}~\label{khintchine}
Let $\U$ and $\mathbf{f}$ be as above. Further assume that 
$\Psi:\mathbb{Z}^{n+1}\setminus \{0\}\rightarrow\bbr_+,$ is a function which satisfies
\begin{itemize}
\item[(i)]$\Psi$ is norm-decreasing, i.e. $\Psi(\q)\ge\Psi(\q')$ for $|\q|_{\infty}\leq|\q'|_{\infty}$,
\item[(ii)]$\sum_{\mathbf{q}\in\bbz^{n+1}\setminus\{0\}}\Psi(\mathbf{q})<\infty$.
\end{itemize}
then the set
$$\mathcal{W}_{\mathbf{f},\Psi}=\{\x\in \mathbf{U} |\hspace{1mm}\mathbf{f}(\x)\hspace{1mm} \mbox{\rm{is}}\hspace{1mm} \Psi-A\}$$
has measure zero. 
\end{thm}

\noindent
Our result in the divergent part is somewhat more restrictive. It is only the $p$-adic case that is proved here i.e. $S$ consists of only one valuation. Note also the function $\Psi$ in theorem~\ref{divergent} is more specific.
%To state the theorem let us fix some further notations. Let $\U$ be an open subset of $\bbq_{\nu}^{d}$ and let $\fbf:\U\rightarrow\bbq_{\nu}^{n}$ be a non-degenerate analytic map. Let $\psi:\bbz\rightarrow\bbr_+$ be a non-increasing function for which $\sum_{k}\psi(k)$ diverges. Define the function $\Psi$ by $$\Psi(\q)=\|\q\|_{\infty}^{-n}\psi(\|\q\|_{\infty})\hspace{3mm}\mbox{for any}\hspace{3mm}\q\in\bbz^{n+1}\setminus\{0\}.$$

\begin{thm}~\label{divergent}
Let $\U$ be an open subset of $\bbq_{\nu}^{d}$ and let $\fbf:\U\rightarrow\bbq_{\nu}^{n}$ be a non-degenerate analytic map (in the above sense). Let $\psi:\bbz\rightarrow\bbr_+$ be a non-increasing function for which $\sum_{k}\psi(k)$ diverges. Define the function $\Psi$ by $$\Psi(\tq)=\|\tq\|_{\infty}^{-n}\psi(\|\tq\|_{\infty})\hspace{3mm}\mbox{for any}\hspace{3mm}\tq\in\bbz^{n+1}\setminus\{0\}$$
Then the set
$$\mathcal{W}_{\mathbf{f},\Psi}=\{\x\in \mathbf{U} |\hspace{1mm}\mathbf{f}(\x)\hspace{1mm} \mbox{\rm{is}}\hspace{1mm} \Psi-A\}$$
has full measure.
\end{thm}
   
\vspace{1cm} 
\begin{remark}~\label{mainremark} \item[1.] Although our result in the divergent case is more restrictive than that of the convergence case it actually is the formulation which has been historically considered. For example Mahler's conjecture for the Veronese curve is formulated in the setting as in theorem~\ref{divergent}. However it is interesting to prove a multiplicative version of the divergence part, this is not known in the real case either.    
\item[2-] Here we just look at simultaneous approximation in non-Archimedean places. As was mentioned above in~\cite{MS}, we proved a convergence Khintchine-type theorem for $\bbq_S$-manifolds, where $S$ contains the infinite place. Indeed, in that case, we defined the notion of $(\Psi, \mathcal{R})$-A for $\mathcal{R}$ a finitely generated subring of $\bbq$ and $\Psi$ a function from $\mathcal{R}^n$ to $\bbr^+$, and prove a convergence $\mathcal{R}$-Khintchine-type theorem. Such definition is not available in the setting here as the interplay between the infinity norm and the finite places is more subtle.
 \item[3-] Most of our proof in the convergence part works for a function $\Psi$ which is only decreasing in the $S$-norm of the coordinates, and not necessarily in the norm of the vector. In fact, in~\cite{MS}, we only assume this weaker assumption. It is interesting to see if the norm-decreasing condition can be relaxed.  
\end{remark}
 
The proof of the convergence part is very similar to the proof in~\cite{MS}. In particular the main technical difficulties that arise in carrying out the strategies developed in~\cite{BKM} to the $S$-arithmetic setting  have the same nature in these two papers. As similar as these two papers are there are many differences in details in both the ``calculus lemma" and ``dynamics part". This and the fact that the divergence case is also treated here demanded a coherent separate paper in the $p$-adic setting. 
 
 %The proofs will essentially follow the same stream lines in~\cite{BKM, BBKM}. However we shall need to do careful analysis on families of $p$-adic $C^k$ functions. Non-Archimedean spaces are ``easier" when one deals with number theoretic properties. However the analysis in some problems gets subtle as we have neither {\it mean value theorem} nor {\it connectedness}! Not having {\it angle} is another problem! So almost all the steps need a new approach or at least perspective. Nevertheless the whole picture is the same. 

Theorem~\ref{khintchine} will be proved with the aid of the following two theorems and applying Borel-Cantelli lemma. The following is what we refered to as the calculus lemma.
 
\begin{thm}~\label{>}
Let $\mathbf{U}$, $\mathbf{f}$ be as in theorem~\ref{khintchine}, $0<\epsilon<\frac{1}{2\kappa}$, for any ball $\mathbf{B}\subset\U$ let $$\mathcal{A}=\left\{\x\in\mathbf{B}|\h\exists\hspace{.5mm}\tilde{\q}\in\bbz^{n+1}\setminus\{0\},\h|\tilde{\q}|_{\infty}<T,\h\begin{array}{l}|\tilde{\q}\cdot\mathbf{f}(\x)|_{S}^{\kappa}\h<\delta T^{-n-1} \\  \|\q\g\mathbf{f}(\x)\|_{\nu}>|\tq|_{\infty}^{-\epsilon}\h{\mbox{for all}}\h\nu\in S\end{array}\right\}.$$ Then we have $|\mathcal{A}|<C \delta\hspace{1mm}|\mathbf{B}|,$ for some universal constant $C.$ 
\end{thm}
The remaining set will be controlled using the following theorem. The proof of this theorm has dynamical nature. 

\begin{thm}~\label{<}
Let $\U$ and $\mathbf{f}$ be as in theorem~\ref{khintchine}. For any
$\x=(x_{\nu})_{\nu\in S}\in \U$, one can find a neighborhood
$\mathbf{V}=\p V_{\nu}\subseteq \U$ of $\x$ and $\alpha>0$ with
the following property: If $\mathbf{B}\subseteq \mathbf{V}$ is any ball then
there exists $E>0$ such that for any choice of $T_0,\cdots,T_n\ge 1$, $K_{\nu}>0$ and $0<\delta^{\kappa}\le\min\frac{1}{T_i}$ where $\delta^{\kappa} T_0\cdots
T_n\p K_{\nu}\le 1$ one has
$$\left|\left\{\x\in\mathbf{B}|\hspace{1mm}\exists
\tilde{\q}\in\bbz^{n+1}\setminus\{0\}:\begin{array}{l}|\tilde{\q}\cdot\mathbf{f}(\x)|_S<\delta\\
\|\g f_{\nu}(x)q\|_{\nu}<K_{\nu}\\|q_i|_{\infty}<T_i\end{array}\right\}\right|\le
E\hspace{.5mm}\vare^{\alpha}|\mathbf{B}|,\hspace{5mm}~(\ref{<})$$ where
$\vare=\max\{\delta,(\delta^{\kappa} T_0\cdots
T_n\p K_{\nu})^{\frac{1}{n+1}})\}.$
\end{thm}

The main idea in the proof of theorem~\ref{divergent} is based on a method of regular systems. This method was first applied in \cite{Ber0} in dimension one and later in \cite{BBKM} it was generalized to any dimension. In our proof we will use the estimates obtained in theorems~\ref{>} and~\ref{<}, and get the following theorem, which provides us with a suitable {\it regular system of resonant sets}. Then a general result on regular systems (cf. Theorem~\ref{approximation}) will kick in and the proof will be concluded. See section~\ref{secregular} for the definitions.

\begin{thm}~\label{regular}
 Let $\fbf$ and $\U$ be as in theorem~\ref{regular}. Given $\tq=(\q,q_0)\in\bbz^n\times\bbz,$ we let 
 $${R}_{\q,q_0}=\{\x\in\U\h|\h\q\cdot\fbf(\x)+q_0=0\}.$$ 
 Define the following set 
 $$\mathcal{R}_{\fbf}=\{R_{\q,q_0}\h|\h (\q,q_0)\in\bbz^n\times\bbz\}$$ 
 and the function $N(R_{\q,q_0})=\|\tq\|_{\infty}^{n+1}.$ Then for almost every $\x_0\in\U$ there is a ball $\mathbf{B}_0\subset\U$ centered at $\x_0$ such that $(\mathcal{R},N,d-1)$ is a regular system in $\mathbf{B}_0.$
\end{thm}

{\bf Structure of the paper.}
In section~\ref{secprelim} we recall some basic geometric and analytic facts about $p$-adic and $S$-arithmetic spaces. Theorem~\ref{>} is proved in section~\ref{sec>}. Section~\ref{secgood} is devoted to the notion of good functions. This section involves only statements of several technical ingredients needed later in the paper. Most of the proofs can be found in~\cite{MS}. Section~\ref{sec<} is devoted to the proof of theorem~\ref{<}. This is actually done modulo theorem~\ref{unipotent} which provides a translation of theorem~\ref{<} into a problem with dynamical nature. This dynamical problem is then solved in section~\ref{secunipotent} using an $S$-arithmetic version of a theorem in~\cite{KM} which was proved in~\cite{KT}. We recall the notion of regular systems in section~\ref{secregular} and prove theorem~\ref{regular} in this section. The proof of the main theorems will be completed in section~\ref{secmain}. The final section~\ref{secremarks} contains some concluding remarks and open problems.

%In section 2, we shall start with some geometry and analysis of $p$-adic spaces, and continue observing some of the properties of discrete $\bbz_{\tilde{S}}$-submodules of $\prod_{\nu\in\tilde{S}} \bbq_{\nu}^{m_{\nu}}$. Section 3 is devoted to the proof of theorem~\ref{>}. In section 4, we recall the notion of good functions and establish the ``goodness" of families of $\nu$-adic analytic functions, which will be needed in the proof of theorem~\ref{<}. This technical section, in some sense, is the core of the proof of the main theorem, which is also used in~\cite{MS}. In section 5, we translate theorem~\ref{<} in terms of recurrence of special flows on the space of discrete $\bbz_{\tilde{S}}$-submodules of $\prod_{\nu\in\tilde{S}} \bbq_{\nu}^{m_{\nu}}$. In section 6, we shall recall a theorem of Kleinbock-Tomanov, and use it to establish theorem~\ref{<} proving its equivalent in the dynamical language. We recall the notion of regular systems in section 7 and prove theorem~\ref{regular} in this section. The proof of the main theorems will be completed in section 8. We shall finish the paper by discussing a few remarks, and open problems.

{\bf Acknowledgments.}
Authors would like to thank G.~A.~Margulis for introducing this topic and suggesting this problem to them. We are also in debt of D.~Kleinbock for reading the first draft and useful discussions. 
%%%%%%%%%%%%%%%%%%%%%%%%%%%%%%%%%%%%%%%%%%%%%%%%%%%%%%%%%%%%%%%%%%%%%%%%%%%
\section{Notations and Preliminary}~\label{secprelim}

%%%%%%%%%%%%%%%%%%%%%%%%%% Orthogonal basis %%%%%%%%%%%%%%%%%%%%%%%%%%%%%%%%

%%%%%%%%%%%%%%%%%%%%%%%%%%%% p-adic calculus %%%%%%%%%%%%%%%%%%%%%%%%%%%%%%%%%%%%%%
\textbf{\textit{Calculus of functions on local fields.}} The terminologies recalled here are from \cite{Sc}. Let
$F$ be a local field and let $f$ be an $F$-valued function defined
on an open subset $U$ of $F.$ Let $$\nabla^k U:=\{(x_1,\cdots,x_k)\in
U^k|\hspace{1mm}x_i\neq x_j
\hspace{1mm}\mbox{\rm{for}}\hspace{1mm} i\neq j\},$$ and define
the $k^{th}$ order difference quotient $\Phi^k f:
\nabla^{k+1}U\rightarrow F$ of $f$ inductively by $\Phi^0 f= f$
and $$\Phi^k
f(x_1,x_2,\cdots,x_{k+1}):=\frac{\Phi^{k-1}f(x_1,x_3,\cdots,x_{k+1})-\Phi^{k-1}f(x_2,x_3,\cdots,x_{k+1})}{x_1-x_2}.$$
As one sees readily $\Phi^k f$ is a symmetric function of its
$k+1$ variables. The function $f,$ is then called $C^k$ at $a\in U$ if the following
limit exits
$$\lim_{(x_1,\cdots,x_{k+1})\rightarrow(a,\cdots,a)} \Phi^k
f(x_1,\cdots,x_{k+1}),$$ and $f$ is called $C^k$ on $U$ if it is $C^k$
at every point $a\in U$. This is equivalent to $\Phi^k f$
being extendable to $\bar{\Phi}^k f:U^{k+1}\rightarrow F.$ This extension, 
if exists, is indeed unique. The $C^k$ functions are $k$ times
differentiable, and
$$f^{(k)}(x)=k!\bar{\Phi}^k(x,\cdots,x).$$ Note
that, $f\in C^k$ implies $f^{(k)}$ is continuous but the converse
fails. Also $C^{\infty}(U)$ is defined to be the class of
functions which are $C^k$ on $U$, for any $k$.  
Analytic functions are indeed $C^{\infty}.$

Let now $f$ be an $F$-valued function of several variables. Denote by
$\Phi_{i}^{k} f$ the $k^{th}$ order difference quotient of $f$
with respect to the $i^{th}$ coordinate. Then for any multi-index
$\beta=(i_1,\cdots,i_d)$ let
$$\Phi_{\beta}f:=\Phi_1^{i_1}\circ\cdots\circ\Phi_d^{i_d} f.$$
One defines the notion of $C^k$ functions correspondingly.

%%%%%%%%%%%%%%%%%%%%%% Discrete submodules of *Q_v^d_v %%%%%%%%%%%%%%%%%%%%%%%%%%%%
\textit{\textbf{$S$-Arithmetic spaces:}} 
Let $\nu$ be any place of $\bbq$ we denote by $\bbq_{\nu}$ the completion of $\bbq$ with respect to $\nu$ and let $\bbq_{S}=\p \bbq_{\nu}.$ If $\nu$ is a finite place we let $p_{\nu}$ be the uniformaizer and $\bbz_{\nu}$ the ring of $\nu$-integers. Given a $\bbq_{\nu}$-vector space $\mathcal{V}$ and a basis $\mathfrak{B}$ we let $\|\h\|_{\mathfrak{B}}$ denote the max norm with respect to this basis and we drop the index $\mathfrak{B}$ from the notation if there is no confusion. This naturally extends to a norm on $\bigwedge\mathcal{V}.$ If $\mathcal{R}$ is any ring and $x,y\in\mathcal{R}^n$ we let $x\cdot  y=\sum_{i=1}^nx^{(i)}y^{(i)}.$ %Let $\mathcal{R}$ be a ring and $\mathcal{V}$ a module over $\mathcal{R}.$ Then for any subset $\mathcal{X}$ of $\mathcal{V}$ we let $\mathcal{X}_{\mathcal{R}}$ denote the $\mathcal{R}$-span of $\mathcal{X}.$ 
The following is the definition of ``orthogonality" which will be useful in the sequel. 

\begin{definition}~\label{orthogonal}
Let $\nu$ be a finite place of $\bbq$. A set of vectors ${x_1, \cdots, x_n}$
in $\bbq _{\nu}^m$, is called orthonormal if  $ \|x_1\| = \|x_2\| = \cdots =
\|x_n\| = \|x_1 \wedge \cdots \wedge x_n\| = 1$, or equivalently when it
can be extended to a $\bbz_{\nu}$-base of $\bbz_{\nu}^m$.
\end{definition}
%\textbf {\textit {Discrete $\bbz_{\tilde{S}}$-submodules of $\prod_{\nu\in\tilde{S}}\bbq_{\nu}^{m_{\nu}}$.}}
Recall that $\bbz_{\tilde{S}}=\bbq\cap\bbq_{\tilde{S}}\cdot\prod_{\nu\not\in S}\bbz_{\nu}$ is co-compact lattice in $\bbq_{\ts},$ where $\bbq$ is embedded diagonally. We normalize the Haar measure so that $\mu_\nu(\bbz_{\nu})=1$ for all finite places and $\mu_{\infty}([0,1])=1$ for the infinite place and we let $\mu$ be the product measure on $\bbq_{\ts}$. With this normalization $\bbz_{\ts}$ has co-volume one in $\bbq_{\ts}.$ For any $\x\in\p\bbq_{\nu}^{m_{\nu}}$, let $c(\x)=\p\|x_{\nu}\|_{\nu}$.  One clearly has $c(\x)\leq\|\x\|_S^{\kappa}.$ The following gives the description of of discrete $\bbz_{\ts}$-modules in $\tp\bbq_{\nu}^{m_{\nu}}.$   

\begin{lem}~\label{base}(cf. ~\cite[Proposition 7.2]{KT})

\noindent
If $\Delta$ is a discrete $\bbz_{\tilde{S}}$-submodule of $\prod_{\nu\in\tilde{S}} \bbq_{\nu}^{m_{\nu}}$,
 then there are $\x^{(1)},\cdots,\x^{(r)}$ in $\prod_{\nu\in\tilde{S}} \bbq_{\nu}^{m_{\nu}}$ so
 that $\Delta=\bbz_{\tilde{S}}\x^{(1)}\oplus\cdots\oplus\bbz_{\tilde{S}}\x^{(r)}$. Moreover $x_{\nu}^{(1)},\cdots,x_{\nu}^{(r)}$ are linearly independent over $\mathbb{Q}_{\nu}$ for any place $\nu\in {\tilde{S}}$
\end{lem}

\begin{definition}~\label{primgood}
Let $\Gamma$ be a discrete $\bbz_{\tilde{S}}$-submodule
of $\prod_{\nu\in\tilde{S}} \bbq_{\nu}^{m_{\nu}}$ then a submodule $\Delta$ of $\Gamma$ is called a
``{\it primitive submodule}" if $\Delta=\Delta_{\bbq_{\tilde{S}}}\cap\Gamma,$ where $\Delta_{\bbq_{\tilde{S}}}$ is the $\bbq_{\tilde S}$-span of $\Delta.$
\end{definition}
\begin{remark}~\label{primgood2} Let $\Gamma$ and $\Delta$ be as in
definition~\ref{primgood} then $\Delta$ is a primitive submodule of $\Gamma$, if and
only if there exists a complementary $\bbz_{\tilde{S}}$-submodule $\Delta'\subseteq \Gamma$, i.e. $\Delta\cap\Delta'=0$ and $\Delta+\Delta'=\Gamma$.
\end{remark}

%%%%%%%%%%%%%%%%%%%%%%%%%%%%%%%%%%%%%%%%%%%%%%%%%%
%%%%%%%%%%%%%%%    PROOF OF THEOREM {>}      %%%%%%%%%%%%%%%%%%%%%%%%%%%%%%%%%%%%%%%%%%%%%%%%%%%%%%%%%%%%%%%%%%%%%%%

\section{Proof of theorem~\ref{>}}~\label{sec>}

\noindent
Fix $\q=(q_1,\cdots,q_n)$ with $|\q|_{\infty}<T$ and let $$\mathcal{A}_{\q}=\left\{\x\in\mathcal{A}|\exists\h q_0, \h |q_0|_{\infty}<T, \hspace{1mm}\begin{array}{l} \mbox{the hypothesis of the theorem}\\ \mbox{holds with}\hspace{2mm} \tq=(q_0,q_1,\cdots,q_n)\end{array} \right\}.$$
We will show that, $|\mathcal{A}_\q|<\delta\h T^{-n}|\mathbf{B}|,$ which then summing over all possible $\q$'s, will finish the proof.
\vspace{1mm}

\noindent
We set $$R=T^{\frac{1}{\kappa}},\hspace{2.5mm}\mbox{and}\hspace{2.5mm}\bbf(\x)=\ball.$$ 
One obviously has $\mathcal{A}_\q\subseteq \cup_{\x\in \mathcal{A}_\q}\bbf(\x).$ Let $\x\in\mathcal{A}_\q,$ \vspace{1mm} then there exists $q_0$ such hat if $\tq=(q_0,\cdots,q_n)$ \vspace{1mm}then $|\tq\cdot\fbf(\x)|_S^{\kappa}<\delta T^{-n-1}.$ We claim that inside $\bbf(\x),$ $q_0$ is the only choice with $|q_0|_{\infty}<T$ such that $\fbf(\bbf(\x)\cap\mathcal{A}_q),$ can be $\frac{1}{4R}$-approximated by $(q_0,q_1,\cdots,q_n).$ 
Indeed for any $\ybf\in\bbf(\x),$ we have $$\tq\cdot\fbf(\ybf)=\tq\cdot\fbf(\x)+(\q\g\fbf(\x))\cdot(\x-\ybf)+\sum_{i,j}\Phi_{ij}(\q\cdot\fbf)\h(x_i-y_i)(x_j-y_j).$$
Comparing the maximum possible values taking into consideration that $\epsilon<\frac{1}{2\kappa},$ we get that $|\tq\cdot\fbf(\ybf)|_S<\frac{1}{4R}$ hence we have $|\tilde{q}\cdot f_{\nu}(x_{\nu})|_{\nu}<\frac{1}{4R}$ for all $\nu\in S.$ To see the claim now notice that if there are $q_0^{1} \h\&\h q_0^{2}$ so that $|\widetilde{q^{(i)}}\cdot f_{\nu}|_{\nu}<\frac{1}{4R}$ we get $\p |q_0^1-q_0^2|_{\nu}<\frac{1}{4T},$ which contradicts the product formula.\vspace{1mm}

We now want to give an upper bound for $|\bbf(\x)\cap\mathcal{A}_\q|,$ where $\x\in\mathcal{A}_\q.$ This will be done using (i) and (ii) below

\vspace{2mm}
\noindent
(i) Pick any $\ybf\in\bbf(\x),$ one has $\|\g q.f_{\nu}(y_{\nu})-\g q.f_{\nu}(x_{\nu})\|_{\nu}<\|\g q.f_{\nu}(x_{\nu})\|_{\nu}/4,$ where $\nu\in S$ is arbitrary. To see this, let $\mathbf{z}=(z_{\nu})$ where $y_{\nu}=x_{\nu}+z_{\nu}$. In this setting, one has $$\partial_i
q\cdot f_{\nu}(y_{\nu})=\partial_i q\cdot f_{\nu}(x_{\nu})+\sum_{j}
\Phi_{j}^{(1)}(\partial_i q\cdot f_{\nu})(t_{\nu}^
{(j)},t_{\nu}^{(j-1)})z_{\nu}^j$$
$$=\partial_i q\cdot f_{\nu}(x_{\nu})+
\sum_{j}(\Phi_{ji}(q\cdot f_{\nu}(t_{\nu}^
{(j)},t_{\nu}^{(j-1)},t_{\nu}^{(j-1)}))+\Phi_{ji}(q\cdot f_{\nu}(t_{\nu}^
{(j)},t_{\nu}^{(j)},t_{\nu}^{(j-1)})))z_{\nu}^j,$$
where $t^{(i)}_{\nu}$'s are coming from the components of $x_{\nu}$ and $y_{\nu}$.
Hence $$|\partial_i q\cdot f_{\nu}(y_{\nu})-\partial_i
q\cdot f_{\nu}(x_{\nu})|_{\nu}<|z_{\nu}|_{\nu}\leq\frac{1}{4R\|\nabla
{q\cdot f}_{\nu}(x_{\nu})\|_{\nu}}\leq \frac{\|\nabla {q\cdot f}_{\nu}(x_{\nu})\|_{\nu}}{4},$$ as we
claimed.

\vspace{1mm}
\noindent
(ii) We now bound $|(\mathcal{A}_\q)_{\nu}\cap\nuball|$ from above for all $\nu\in S.$ Without loss of generality we may assume $|q\g f_{\nu}(x_{\nu})|_{\nu}=|q\cdot\d_1f_{\nu}(x_{\nu})|_{\nu}.$ Let $\ybf\in\bbf(\x)\cap\mathcal{A}_\q,$ then we have $$\tilde{q}\cdot f_{\nu}(y_{\nu}+\alpha e_1)-\tilde{q}\cdot f_{\nu}(y_{\nu})=q\cdot\d_1f_{\nu}(y_{\nu})\alpha+\Phi_{11}q\cdot f(y_{\nu}+\alpha e_1,y_{\nu},y_{\nu})\alpha^2.$$ As before a norm comparison, using the fact $\epsilon<\frac{1}{2\kappa}$ gives us 
\begin{equation}~\label{>1}|\tilde{q}\cdot f_{\nu}(y_{\nu}+\alpha e_1)-\tilde{q}\cdot f_{\nu}(y_{\nu})|_{\nu}=|q\cdot\d_1f_{\nu}(y_{\nu})|_{\nu}|\alpha|_{\nu}\h \end{equation}
Now set $\widehat{{y}_{\nu}}=(\widehat{y_{\nu}^{(1)}},y_{\nu}^{(2)},\cdots,y_{\nu}^{(d_{\nu})})$
for fixed $\{y_{\nu}^2,\cdots,y_{\nu}^{d_{\nu}}\}.$ Using~(\ref{>1}) we have the measure of
$$(\mathcal{A}_\q)_{\widehat{{y}_{\nu}}}=\{y_{\nu}^1\in\bbq_{\nu}|\hspace{1mm}
{y}_{\nu}=(y_{\nu}^{(1)},\widehat{y_{\nu}})\in
\nuball\cap(\mathcal{A}_\q)_{\nu}\}$$ is at most
$\frac{C''\h(\delta T^{-n-1})^{\frac{1}{\kappa}}}{\|\g q\cdot f_{\nu}(x_{\nu})\|_{\nu}},$ where $C'$ is a universal constant. This gives us $$|(\mathcal{A}_\q)_{\nu}\cap\nuball|\leq C'\delta^{\frac{1}{\kappa}} R\h T^{\frac{-n-1}{\kappa}}|\nuball|.$$ 

One uses now the fact that for non-Archimedean valuations two balls are either disjoint or one contains the other and gets $|(\mathcal{A}_\q)_{\nu}|\leq C'\delta^{\frac{1}{\kappa}} R\h T^{\frac{-n-1}{\kappa}}|\mathbf{B}_{\nu}|.$ Multiplying these inequalities for various $\nu$'s we get $$|\mathcal{A}_\q|\leq C\delta T^{-n}|\mathbf{B}|.$$
We now sum up over all possible $\q$'s and get $|\mathcal{A}|\leq C\delta|\mathbf{B}|,$ as we wished. 

%%%%%%%%%%%%%%%%%%%%%%%%%%%%%%%%%%%%%%%%%%%%%%%%%%%
%%%%%%%%%%%%%%%%%       GOOD FUNCTIONS     %%%%%%%%%%%%%%%%%%%%%
%%%%%%%%%%%%%%%%%%%%%%%%%%%%%%%%%%%%%%%%%%%%%%%%%%%

\section{Good Functions}~\label{secgood}
In this section we will state conditions which guarantee the ``polynomial like" behavior of certain classes of maps on local fields. This notion which we refer to as ``{\it good function}" generalizes the class of functions considered in~\cite{EMS}. The definition was suggested in~\cite{KM}. In what follows we just recall statements which are needed in the course of proof of theorem~\ref{unipotent}. The proofs can be found in~\cite{MS}.

\begin{definition}~\label{defgood}(cf. \cite[Section 3]{KM}) 

\noindent
Let $C$ and $\alpha$ be positive real numbers. A function $\fbf$
defined on an open set $\mathbf{V}$ of $X=\prod_{\nu\in S}
\bbq_{\nu}^{m_\nu}$ is called $(C, \alpha)$-good, if for any open
ball $\mathbf{B}\subset \mathbf{V}$ and any $\vare >0$ one has
$$|\{\mathbf{x}\in \mathbf{B}|\hspace{1mm}
\|\mathbf{f}(\mathbf{x})\|< \vare\cdot\sup_{\mathbf{x}\in
\mathbf{B}}\|\mathbf{f}(\mathbf{x})\|\}|\leq C
\hspace{1mm}\vare^{\alpha}|\mathbf{B}|.$$ 
\end{definition} 

\begin{remark}~\label{remarkgood}
The following are consequences of the definition. Let $X, \mathbf{V}$ and $\mathbf{f}$ be as in definition~\ref{defgood}. Then
\begin{itemize}
\item[(i)] $\mathbf{f}$ is $(C,\alpha)$-good on $\mathbf{V}$ if and only if $\|\mathbf{f}\|$ is $(C,\alpha)$-good.

\item[(ii)] If $\mathbf{f}$ is $(C,\alpha)$-good on $\mathbf{V}$,
then so is $\lambda\mathbf{f}$ for any $\lambda\in \bbq_{S}.$

\item[(iii)] Let $I$ be an index set, if $\mathbf{f}_i$ is $(C,\alpha)$-good on
$\mathbf{V}$ for any $i\in I$, then so is $\sup_{i\in
I}\|\mathbf{f}\|.$

\item[(iv)] If $\mathbf{f}$ is $(C,\alpha)$-good on $\mathbf{V}$
and $c_1\leq
\|\mathbf{f}(\mathbf{x})\|_{S}/\|\mathbf{g}(\mathbf{x})\|_{S}\leq
c_2$, for any $x\in \mathbf{V},$ then $\mathbf{g}$ is
$(C({c_2}/{c_1})^{\alpha},\alpha)$-good on $\mathbf{V}.$
\end{itemize}
\end{remark}

As we mentioned above, the definition of  good functions seeks for a polynomial like behavior of maps. The next lemma guarantees that indeed polynomials are good.

\begin{lem}~\label{polygood}(cf.~\cite[Lemma 2.4]{KT})

\noindent
Let $\hspace{1mm}\nu$ be any place of $\bbq$ and $p\in\bbq_{\nu}[x_1,\cdots,x_d]$
 be a polynomial of degree not greater than $l$. Then there exists
 $\hspace{1mm}C=C_{d, l}\hspace{1mm}$ independent of $p$, such that $p$ is $(C,1/{dl})$-good on $\bbq_{\nu}$.
\end{lem}

Next theorem ``relates" the definition of a good function to conditions on its derivatives. 
 
\begin{thm}~\label{good}(cf.~\cite[Theorem 4.4]{MS} )
Let $V_1, \cdots, V_d$ be nonempty open sets in $\bbq_{\nu}.$ Let
$k\in \bbn$, $A_1, \cdots, A_d, A'_1, \cdots, A'_d$ be positive real
numbers and $f\in C^k(V_1\times \cdots \times V_d)$ be such that
$$A_i\leq |\Phi_i^{k}f|_{\nu}\leq A'_i \hspace{2mm}\mbox{\rm{on}}
\bigtriangledown^{k+1}V_i\times \prod_{j\neq i} V_j,
\hspace{1mm}i=1\cdots, d .$$ Then $f$ is $(C, \alpha)$-good on
$V_1\times\cdots\times V_d$, where $C$ and $\alpha$ depend only
on $k, d, A_i$, and $A'_i$ .
\end{thm}
\noindent

We need to show some families of functions are good with uniform constants. The following gives a condition to guarantee such assertion. The proof of this uses compactness arguments and theorem~\ref{good} above. In our setting we actually will use the proceeding corollary.  

\begin{thm}~\label{compact}(cf.~\cite[theorem 4.5]{MS})

\noindent
Let $U$ be an open neighborhood of $x_0\in\bbq_{\nu}^m$ and let
$\mathcal{F}\subset C^l(U)$ be a family of functions $f : U
\rightarrow \bbq_{\nu}$ such that
\begin{itemize}
\item[1.] ${\nabla f\in \mathcal{F}}$ is compact in $C^{l-1}(U)$

\item[2.] $ \inf_{f\in \mathcal{F}} \sup_{|\beta| \leq l}|\partial_{\beta}f(x_0)|> 0.$
\end{itemize}
Then there exist a neighborhood $V\subseteq U$ of $x_0$ and
positive numbers $C=C(\mathcal{F})$ and
$\alpha=\alpha(\mathcal{F})$ such that for any $f\in \mathcal{F}$
\begin{itemize}
\item[(i)] $f$ is $(C, \alpha)$-good on $V$.

\item[(ii)] $\nabla f$ is $(C, \alpha)$-good on $V$.
\end{itemize}
\end{thm}

\begin{cor}~\label{linear}
Let $f_1, f_2, \cdots, f_n$ be analytic functions from a neighborhood $U$ of $x_0$ in $\bbq_{\nu}^m$ to $\bbq_{\nu}$, such that $1,\hspace{1mm}f_1, f_2, \cdots, f_n$ are linearly independent on any neighborhood of $x_0$, then
\begin{itemize}
\item[(i)] There exist a neighborhood $V$ of $x_0$, $C \hspace{1mm}\& \hspace{1mm}\alpha>0$ such that any linear combination of $\hspace{1mm}1, f_1, f_2, \cdots, f_n$ is $(C, \alpha)$-good on $V$.

\item[(ii)] There exist a neighborhood $V'$ of $x_0$, $C' \hspace{1mm}\& \hspace{1mm}\alpha'>0$ such that for any $\hspace{1mm}c_1, c_2, \cdots, c_n \in \bbq_{\nu}$, $\|\sum_{k=1}^{n} c_i \nabla f_i\|$ is $(C', \alpha')$-good.
\end{itemize}
\end{cor}
\noindent
We now recall the notion of {\it skew gradient} from \cite[Section 4] {BKM}. For $i=1,2$ let $g_i:\mathbb{Q}_{\nu}^d\rightarrow\mathbb{Q}_{\nu}$ be two $C^1$ functions. Define then $\widetilde\g(g_1,g_2):=g_1\g g_2-g_2\g g_1.$ This, in some sense, measures how far two functions are from being linearly dependent. The following is the main technical result of section 4 in~\cite{MS}. Let us remark that this theorem is responsible for the fact that the results of this paper are in the setting of analytic functions rather than $C^k$ functions. In Archimedean case the proof of this fact uses polar coordinates which is not available in non-Archimedean setting. 

\begin{thm}~\label{tartibat}(cf.~\cite[Theorem 4.7]{MS})

\noindent
Let $U$ be a neighborhood of $x_0\in \bbq_{\nu}^m$ and $f_1, f_2,
\cdots, f_n$ be analytic functions from $U$ to $\bbq_{\nu}$, such that $1,\hspace{.5mm}f_1, f_2,\cdots, f_n$ are linearly independent on any open subset of $U.$ Let $F=(f_1, \cdots,
f_n)$ and $$\mathcal{F}= \{(D_1\cdot F,\hspace{1mm} D_2\cdot F+a)|\hspace{1mm}
\|D_1\|=\|D_2\|=\|D_1\wedge D_2\|=1,\hspace{1mm}
 D_1, D_2\in \bbq_{\nu}^n,\hspace{1mm} a\in \bbq_{\nu}\}.$$
 Then there exists a neighborhood $V\subseteq U$ of $x_0$ such that
\begin{itemize}
\item[(i)] For any neighborhood $B\subseteq V$ of $x_0$,
 there exists $\rho=\rho(\mathcal{F}, B)$ such that
 $\sup_{x\in B} \parallel \widetilde{\nabla} g(x)\parallel \geq \rho$ for any $g\in \mathcal{F}$.

\item[(ii)] There exist $C$ and $\hspace{1mm}\alpha,$ positive
numbers, such that for any $g\in \mathcal{F},$ $\| \widetilde{\nabla} g\|$ is $(C, \alpha)$-good on $V.$ 
\end{itemize}
\end{thm}

%%%%%%%%%%%%%%%%%%%%%%%%%%%%%%%%%%%%%%%%%%%%%%%%%%%
%%%%%%%%%%%%%%%%%%%%%%%%%%%%%%%%%%%%%%%%%%%%%%%%%%%
%%%%%%%%%%%%%%%%%%%%%%%%%%%%%%PROOF OF (<)%%%%%%%%%%%%%%%%%%%%%%%%%%%%%%%%
\section{Theorem~\ref{<} and lattices}~\label{sec<}
In this section we prove theorem~\ref{<}. This is done with the aid of converting the problem into a question about quantitative recurrence properties of some ``special flows" on the space of discrete $\bbz_{\ts}$-modules. This dynamical translation was the break through by Kleinbock and Margulis, see in~\cite{KM}. This point of view was then followed in~\cite{BKM}, ~\cite{KT} and~\cite{MS}.
\vspace{1mm}
 
Till now we essentially worked in a single non-Archimedean place. From this point on we need to work with all the places in $\tilde{S}=\{\infty\}\cup S$, simultaneously. Let us fix some further notations to be used in the sequel. We consider the zero function on the real component and set $d_{\infty}=0$. Let $m_{\nu}=n+d_{\nu}+1$ for all $\nu\in \tilde{S}.$ Let $\{e_{\nu}^{0},e_{\nu}^{*1},\cdots,e_{\nu}^{*d_{\nu}},e_{\nu}^1,\cdots,e_{\nu}^n\}$ be the standard basis for $\bbq_{\nu}^{m_{\nu}}.$ Now $\e_i=(e_{\nu}^i)_{\nu\in\tilde{S}}$ gives a basis for $X=\tp \bbq_{\nu}^{m_{\nu}}$ as a $\bbq_{\ts}$-module. %Let $W_{\nu}$ (resp. $W_{\nu}^*$) be the $\bbq_{\nu}$-span of  $\{e_{\nu}^1,\cdots,e_{\nu}^{n-1}\}_{\bbq_{\nu}}$ (resp. $\{e_{\nu}^{*1},\cdots,e_{\nu}^{*d_{\nu}}\}$). 
Define the $\bbz_{\ts}$-module $\Lambda$ to be the $\bbz_{\ts}$-span of $\{\e_0,\cdots,\e_n\}.$ Let $\fbf$ be as in the statement of theorem~\ref{<}. Then for any $\x\in\mathbf{U}$ define 

%we will be interested in the $\bbq_{\tilde{S}}\hspace{.5mm}$-module $X=\tp \bbq_{\nu}^{m_{\nu}}$. We denote by
%$\{e_{\nu}^{0},e_{\nu}^{*1},\cdots,e_{\nu}^{*d_{\nu}},e_{\nu}^1,\cdots,e_{\nu}^n\},$ the
%standard basis of the $\nu$-factor $\bbq_{\nu}^{m_{\nu}}$ of $X.$ 
%We define the following spaces
%$W_{\nu}^*=\{e_{\nu}^{*1},\cdots,e_{\nu}^{*d_{\nu}}\}_{\bbq_{\nu}}$,
%$W_{\nu}=\{e_{\nu}^1,\cdots,e_{\nu}^{n-1}\}_{\bbq_{\nu}}$, and let
%$\Lambda$ be the discrete $\bbz_{\tilde{S}}\hspace{.5mm}$-module generated by
%$\e_0,\cdots,\e_n$,
%where $\e_i=(e_{\nu}^i)_{\nu\in\tilde{S}}$ for any $0\le i\le n.$ Let the function $\fbf$ and the constants $\delta$, $K_{\nu}$'s, $T_i$'s be as in theorem~\ref{<}, and let

$$\u_{\x}=\left(\left(\begin{array}{ccc}1&0&f_{\nu}(x_{\nu})\\0&I_{d_{\nu}}&\g f_{\nu}(x_{\nu})
\\0&0&I_n\end{array}\right)\right)_{\nu\in \tilde{S}}$$ 
Note that in the real place we have the identity matrix $I_{n+1}.$ If $0_{d_{\nu}}$ denotes the $d_{\nu}\times1$ zero block then one has
$$\u_{\x}\left(\left(\begin{array}{l}p\\0_{d_{\nu}}\\\vec{q}\end{array}\right)\right)_{\nu\in \tilde{S}}
=\left(\left(\begin{array}{c}p+f_{\nu}(x_{\nu})\cdot \vec{q}\\\g f_{\nu}(x_{\nu}) \vec{q}\\
\vec{q}\end{array}\right)\right)_{\nu\in \tilde{S}}.$$ 
Let $\vare>0$ be given. Define the diagonal matrix
$$\mathbf{D}=(D_{\nu})_{\nu\in\tilde{S}}=(\diag((a_{\nu}^{(0)})^{-1},(a_{\nu}^*)^{-1},
\cdots,(a_{\nu}^*)^{-1},(a_{\nu}^{(1)})^{-1},\cdots,(a_{\nu}^{(n)})^{-1}))_{\nu\in\ts}$$ 
where $a_{\nu}^{(0)}=\begin{cases}\lceil\delta\rceil_{\nu}\hspace{3mm}\nu\in S\\T_0/\vare\hspace{3mm}\nu=\infty\end{cases}, a_{\nu}^{*}=\lceil K_{\nu}\rceil_{\nu},$
$a_{\nu}^{(i)}=\begin{cases}1\hspace{7mm}\nu\in S\\T_i/\vare\hspace{2mm}\nu=\infty\end{cases}$ $1\le i\le n,$ and for a positive real
number $a$ and $\nu\in S$ we let $\lceil
a\rceil_{\nu}$ (resp. $\lfloor a\rfloor_{\nu}$) denote a power of
$p_{\nu}$ with the smallest (resp. largest) $\nu$-adic norm bigger
(resp. smaller) than $a$. The constants $\delta, K_{\nu}$ and $T_i$ above are as in the statement of theorem~\ref{<}.

%So if $\lambda=\left(\left(\begin{array}{l}p
%\\0_{d_{\nu}}\\\vec{q}\end{array}\right)\right)_{\nu\in\tilde{S}}$ has been chosen such that $(p,\vec{\q})$ satisfies the conditions on the set~\ref{<}, we get an upper bound on each
%of the coordinates of $\u_{\x}\lambda$. We now want to rescale the space to have $c(\mathbf{D}\u_{\mathbf{y}}\lambda)<\vare.$ This is done by multiplying the space with the diagonal matrix 
%$$\mathbf{D}=(D_{\nu})_{\nu\in\tilde{S}}=(\diag((a_{\nu}^{(0)})^{-1},(a_{\nu}^*)^{-1},
%\cdots,(a_{\nu}^*)^{-1},(a_{\nu}^{(1)})^{-1},\cdots,(a_{\nu}^{(n)})^{-1}))_{\nu\in S}$$ 
 %where $a_{\nu}^{(0)}=\begin{cases}\lceil\delta\rceil_{\nu}\hspace{3mm}\nu\in S\\T_0/\vare\hspace{3mm}\nu=\infty\end{cases}, a_{\nu}^{*}=\lceil K_{\nu}\rceil_{\nu},$
%$a_{\nu}^{(i)}=\begin{cases}1\hspace{7mm}\nu\in S\\T_i/\vare\hspace{2mm}\nu=\infty\end{cases}$ $1\le i\le n.$ 
%here $\lceil a\rceil_{\nu}$ denotes a power of
%$p_{\nu}$ with the smallest $\nu$-adic norm bigger than $a$.\vspace{1mm} 
%This is done to make $c(\mathbf{D}\u_{\mathbf{y}}\lambda)<\vare,$ where $c(\ybf)=\tp |\ybf|_{\nu}$ for any $\ybf\in\bbq_{\ts}^n.$
The following, which will be proved in section~\ref{secunipotent}, proves theorem~\ref{<}.

\begin{thm}~\label{unipotent}
Let $\mathbf{U}$ and $\mathbf{f}$ be as in theorem~\ref{<}; then
for any $\x=(x_{\nu})_{\nu\in S}$, there exists a neighborhood
$\mathbf{V}=\p V_{\nu}\subseteq \mathbf{U}$ of $\x$, and a
positive number $\alpha$ with the following property: for any
$\mathbf{B}\subseteq\mathbf{V}$ there exists $E>0$ such that for any
$\mathbf{D}=(\diag((a_{\nu}^{(0)})^{-1},(a_{\nu}^*)^{-1},
\cdots,(a_{\nu}^*)^{-1},(a_{\nu}^{(1)})^{-1},\cdots,(a_{\nu}^{(n)})^{-1}))_{\nu\in
\tilde{S}}$ with $1\leq|a_{\infty}^{(i)}|_{\infty},$ $0<|a_{\nu}^{(0)}|_{\nu}\le 1\le
|a_{\nu}^{(1)}|_{\nu}\le\cdots\le |a_{\nu}^{(n)}|_{\nu}$ for all $\nu\in S$ which satisfy
\begin{itemize}
\item[(i)]$0<\p|a_{\nu}^*|_{\nu}\le |a_{\infty}^{(0)}a_{\infty}^{(1)} \cdots
a_{\infty}^{(n)}|_{\infty}^{-1}\p|a_{\nu}^{(0)}a_{\nu}^{(1)} \cdots
a_{\nu}^{(n-1)}|_{\nu}^{-1}.$
\item[(ii)]$1\leq\min_i\left|\frac{1}{a_{\infty}^{(i)}}\right|_{\infty}\p|a_{\nu}^{(0)}|_{\nu}^{-1}$
\end{itemize}
and for any positive number $\vare$,
one has
$$|\{\mathbf{y}\in\mathbf{B}|\hspace{1mm}c(\mathbf{D}\u_{\mathbf{y}}\lambda)<
\vare\hspace{1mm}\mbox{\rm{for some}}\hspace{1mm}\lambda\in\Lambda\setminus\{0\}\}|
\le E\hspace{1mm}\vare^{\alpha}|B|.$$
\end{thm}
\begin{proof}[Proof of theorem~\ref{<} modulo theorem~\ref{unipotent}]
Choose $\vare$ as in theorem~\ref{<} and define $a_{\nu}^{(i)}$'s, $a_{\nu}^*$ and $\mathbf{D}$ as above. Our assumptions in theorem~\ref{<} guarantee that $\mathbf{D}$ satisfies the conditions above. Now if $\lambda=\left(\left(\begin{array}{l}p
\\0_{d_{\nu}}\\\vec{q}\end{array}\right)\right)_{\nu\in\tilde{S}}$  is \vspace{.75mm}such that $(p,\vec{\q})$ satisfies the conditions in~\ref{<} then we have $c(\mathbf{D}\u_{\mathbf{y}}\lambda)<\vare.$ Recall that $c(\x)\leq \|\x\|_S^{\kappa},$ now $\mathbf{V}$ and $\alpha/\kappa$ as in theorem~\ref{unipotent} satisfy conditions of theorem~\ref{<}.

%It is easy to verify that if one defines $a_{\nu}^{(i)}$'s and $a_{\nu}^*$ as in the setting
%of beginning of this section, they satisfy conditions of theorem~\ref{unipotent}.
%Hence theorem~\ref{unipotent} provides us with a neighborhood $\mathbf{V}$ and a positive number $\alpha$. Using the discussion in the beginning of
%this section and the fact that $c(\x)\leq \|\x\|_S^{\kappa}$, one sees that ${\alpha}/{\kappa}$ and $\mathbf{V}$ satisfy the conditions of theorem~\ref{<}.
\end{proof}

%%%%%%%%%%%%%%%%%%%%%%%%%%%%%%%%%%%%%%%%%%%%%%%%%%%
%%%%%%%%%%%%%%%%         PROOF OF UNIPOTENT       %%%%%%%%%%%%%%%%%%%%%%%%%%%%%%%%%%%%%%%%%%%%%%%%%%%%%%%%%%%%%%%%%%%%%

\section{Proof of theorem~\ref{unipotent}}~\label{secunipotent}
In the previous section we reduced the proof of theorem~\ref{<} to theorem~\ref{unipotent}. This section contains the proof of the latter. Theorem~\ref{unipotent} is a far reaching quantitative generalization of recurrence properties of unipotent flows on homogenous spaces. We refer to~\cite{KM} for further discussion and complementary remarks. Let us start with the following   

\begin{definition}~\label{norm-like}(cf. ~\cite[Section 6]{KT})

\noindent
Let $\Omega$ be the set of all discrete $\bbz_{\tilde{S}}$-submodules of $\tp\bbq_{\nu}^{m_{\nu}}$. A function $\theta$ from $\Omega$ to the positive real numbers is called a {\it norm-like map} if the following three properties hold:
\begin{itemize}
\item[i)] For any $\Delta, \Delta'$ with $\Delta'\subseteq\Delta$ and the same $\bbz_{\tilde{S}}$-rank, one has $\theta(\Delta)\leq\theta(\Delta')$.
\item[ii)] For any $\Delta$ and $\gamma\not\in\Delta_{\bbq_{\tilde{S}}}$, one has $\theta(\Delta+\bbz_{\tilde{S}}\gamma)\leq\theta(\Delta)\theta(\bbz_{\tilde{S}}\gamma)$.
\item[iii)] For any $\Delta$, the function $g\mapsto\theta(g\Delta)$ is a continuous function of\\ $g\in\GL(\tp\bbq_{\nu}^{m_{\nu}})$. 
\end{itemize}
\end{definition}

\begin{thm}~\label{poset}(cf.~\cite[Theorem 8.3]{KT})

\noindent
Let $\mathbf{B}=\mathbf{B}(\mathbf{x}_0,r_0)\subset\prod_{\nu\in{S}} \bbq_{\nu}^{d_\nu}$ and $\widehat{\mathbf{B}}=\mathbf{B}(\mathbf{x}_0,3^m r_0)$ for $m=\min_{\nu}{(m_{\nu})}.$ Assume that $\mathbf{H}:\widehat{\mathbf{B}}\rightarrow \rm{GL}(\tp \bbq_{\nu}^{m_\nu})$ is a continuous map. Also let $\mathbf{\theta}$ be a norm-like map defined on the set $\Omega$ of discrete $\bbz_{\tilde{S}}$-submodules of $\tp\bbq_{\nu}^{m_\nu},$ and $\mathfrak{P}$ be a subposet of $\Omega$. For any $\Gamma\in\mathfrak{P}$ denote by $\psi_{\Gamma}$ the function $\mathbf{x}\mapsto \mathbf{\theta}(\mathbf{H}(\mathbf{x})\Gamma)$ on $\widehat{\mathbf{B}}.$ Now suppose for some $C,\alpha>0$ and $\rho>0$ one has
\begin{itemize}
\item[(i)] for every $\Gamma\in\mathfrak{P},$ the function $\psi_{\Gamma}$ is $(C,\alpha)$-good on $\widehat{\mathbf{B}};$

\item[(ii)] for every $\Gamma\in\mathfrak{P},\hspace{1mm}\sup_{\mathbf{x}\in \mathbf{B}}\|\psi_{\Gamma}(\mathbf{x})\|_{\tilde{S}}\geq\rho;$

\item[(iii)] for every $\mathbf{x} \in \widehat{\mathbf{B}},\hspace{2mm} \#\{\Gamma\in\mathfrak{P} |\hspace{1mm}\|\psi_{\Gamma}(\mathbf{x})\|_{\tilde{S}}\leq\rho\}<\infty.$
\end{itemize}
Then for any positive $\vare\leq\rho$ one has $$|\{\mathbf{x}\in
\mathbf{B}|\hspace{1mm}\mathbf{\theta}(\mathbf{H}\mathbf{(x)}\lambda)<\vare
\hspace{1mm}\mbox{\rm{for
some}}\hspace{1mm}\lambda\in\Lambda\smallsetminus\{0\}\}|\leq
mC(N_{((d_{\nu}),S)}D^2)^m{(\frac{\vare}{\rho})}^{\alpha}|\mathbf{B}|,$$
where $D$ may be taken to be $\hspace{1mm}
\p(3p_{\nu})^{d_{\nu}},$ and $N_{((d_{\nu}),S)}$ is the Besicovich constant for the space $\p \mathbb{Q}_{\nu}^{d_{\nu}}.$

\end{thm}

The idea of the proof of theorem~\ref{poset} is very similar to Margulis's proof of recurrence properties of unipotent flows on homogenous spaces, but the proof is more technical. We will prove theorem~\ref{unipotent} using this theorem. However we need to set the stage for using this theorem.

\noindent
\textbf{\textit{The poset:}} let $\Lambda$ be as in section~\ref{sec<} and let $\pfr$ be the poset of primitive $\bbz_{\ts}$-modules of $\Lambda.$ 

\noindent
\textbf{\textit{The norm-like map:}} For any $\nu\in\ts$ we let $\mathcal{I}^{*}_{\nu}$ be the ideal
generated by $\{e_{\nu}^{*i}\wedge e_{\nu}^{*j}\h\mbox{for}\h1\le i,j\le
d_{\nu}\}.$ Note that $\mathcal{I}^{*}_{\infty}=0.$ Let  $\pi_{\nu}:\bigwedge\bbq_{\nu}^{m_{\nu}}\rightarrow\bigwedge\bbq_{\nu}^{m_{\nu}}/\mathcal{I}^{*}_{\nu}$ be the natural projection. For $\x\in\tp\bigwedge\tp\bbq_{\nu}^{m_{\nu}}$ define $\theta(\x)=\tp\theta_{\nu}(x_{\nu})$ where $\theta_{\nu}(x_{\nu})=\para
\pi_{\nu}(x_{\nu})\para_{\pi_{\nu}(\bfr_{\nu})}$ and $\mathfrak{B}_{\nu}$ is the standard basis of $\bigwedge\bbq_{\nu}^{m_{\nu}}.$ Finally for any discrete $\bbz_{\tilde{S}}$-submodule $\Delta$ of $\tp\bbq_{\nu}^{m_{\nu}}$, let $\theta(\Delta)=\theta(\x^{(1)}\wedge\cdots\wedge\x^{(r)})$,
where $\{\x^{(1)},\cdots,\x^{(r)}\}$ is a $\bbz_{\tilde{S}}$-basis of $\Delta$. Using the product formula, it is readily seen that $\theta(\Delta)$ is well-defined. This is our norm-like map.

\vspace{.75mm}
%To this end, we need to define a poset
%$\pfr$, a norm-liked map $\theta$, a family $\mathcal{H}$ of functions, and verify the conditions of theorem~\ref{poset}
%for our choices of $\mathfrak{P}$, $\theta$, and any function $\bf{H}$ in $\mathcal{H}$.
%We start with introducing a norm-like map
%$\theta$ from $\tp\bigwedge \bbq_{\nu}^{m_{\nu}}$ to $\bbr^{+}$, and
%then ``restrict" it to the poset of discrete $\bbz_{\tilde{S}}$-submodules of
%$\tp\bbq_{\nu}^{m_{\nu}}$. For each $\nu\in\tilde{S}$ let $\mathcal{I}^{*}_{\nu}$ be the ideal
%generated by $e_{\nu}^{*i}\wedge e_{\nu}^{*j}$, for $1\le i,j\le
%d_{\nu}$, and $\pi_{\nu}$ be the natural map from
%$\bigwedge\bbq_{\nu}^{m_{\nu}}$ to
%$\bigwedge\bbq_{\nu}^{m_{\nu}}/\mathcal{I}^{*}_{\nu}$, indeed for $\nu=\infty$ we just have $\mathcal{I}_{\infty}^{*}=0$. Define
%$\theta_{\nu}(x_{\nu})=\para
%\pi_{\nu}(x_{\nu})\para_{\pi_{\nu}(\bfr_{\nu})}$, where $\mathfrak{B}_{\nu}$ is
%the standard basis of $\bigwedge\bbq_{\nu}^{m_{\nu}}$, and let
%$\theta(\x)=\tp\theta_{\nu}(x_{\nu})$. For any
%discrete $\bbz_{\tilde{S}}$-submodule $\Delta$ of $\tp\bbq_{\nu}^{m_{\nu}}$,
%let $\theta(\Delta)=\theta(\x^{(1)}\wedge\cdots\wedge\x^{(r)})$,
%where $\{\x^{(1)},\cdots,\x^{(r)}\}$ is a $\bbz_{\tilde{S}}$-basis of $\Delta$.
%Using the product formula, it is seen readily that $\theta(\Delta)$ is well-defined, and it is a norm-like map.

%\noindent 
%Now let $\pfr$ be the poset of primitive $\bbz_{\tilde{S}}$-submodules of
%$\Lambda,$ where $\Lambda$ is defined in section $4$. 
\noindent
\textbf{\textit{The family $\mathcal{H}$:}}
Let $\mathcal{H}$ be the family of functions
$$\bf{H}:\mathbf{U}=\tp U_{\nu}\rightarrow {\rm
GL}(\tp\bbq_{\nu}^{m_{\nu}})\hspace{3mm}\mbox{\rm{where}}\hspace{3mm}
\bf{H}(\x)=\mathbf{D}\u_{\x},$$ 
where $\mathbf{D}$ and $\mathcal{U}_{\x}$ are as in theorem~\ref{unipotent}.

Note that the restriction of $\theta$ to $\tp\bbq_{\nu}^{m_{\nu}}$ is the same as the function $c$. Hence theorem~\ref{poset} reduces the proof of theorem~\ref{unipotent} to finding a neighborhood $\mathbf{V}$ of $\x$ which satisfies the following
\begin{itemize}
\item[(I)] There exist $C,\alpha>0$, such that all the functions
$\mathbf{y}\mapsto\theta(\mathbf{H} (\ybf)\Delta)$, where $\bf{H}\in\mathcal{H}$
 and $\Delta\in\pfr$ are $(C,\alpha)$-good on $\mathbf{V}$.
\item[(II)]For all $\ybf\in\mathbf{V}$ and $\bf{H}\in\mathcal{H}$, one
has
$\#\{\Delta\in\pfr|\hspace{1mm}\theta(\mathbf{H} (\ybf)\Delta)\le1\}<\infty.$
\item[(III)] For every ball $\mathbf{B}\subseteq\mathbf{V}$,
 there exists $\rho>0$ such that $\sup_{\ybf\in\bbf}\theta(\mathbf{H}(\ybf)\Delta)\ge\rho$
 for all $\bf{H}\in\mathcal{H}$ and $\Delta\in\pfr$.
\end{itemize}

If $\x\in\p\bbq_{\nu}^{d_{\nu}}$ define $\mathbf{V}=\p V_{\nu},$ where $V_{\nu}$ is small enough such that assertions of corollary~\ref{linear} and theorem~\ref{tartibat} hold.  We now verify (I), (II), (III) for this choice of $\mathbf{V}$.
\vspace{2mm}

\noindent \textit{\textbf{Proof of (I).}} Let $\Delta$ be a primitive submodule of $\Lambda$ and let $k=\mbox{rank}_{\bbz_{\tilde{S}}}\Delta.$ Denote by
$(\dbf\Delta)_{\nu}$ the $\mathbb{Q}_{\nu}$-span of the projection of $\dbf\Delta$ to the place $\nu\in\ts,$ note that $\dim_{\bbq_{\nu}}(\dbf\Delta)_{\nu}=k$. Let $W_{\nu}$ (resp. $W_{\nu}^*$) be the $\bbq_{\nu}$-span of  $\{e_{\nu}^1,\cdots,e_{\nu}^{n-1}\}_{\bbq_{\nu}}$ (resp. $\{e_{\nu}^{*1},\cdots,e_{\nu}^{*d_{\nu}}\}$). Let $x_{\nu}^{(1)},\cdots,x_{\nu}^{(k-1)}\in (\dbf\Delta)_{\nu}\cap
W_{\nu}\oplus\bbq_{\nu}e_{\nu}^{n}$ be an orthonormal set, see section~\ref{secprelim} for the definition in the non-arthimedean setting. Complete this to an orthonormal basis for $(\dbf\Delta)_{\nu}\oplus\bbq_{\nu}e_{\nu}^0$ by adding $e_{\nu}^0$ and $x_{\nu}^{(0)}$ if needed. Let $\{\ybf^{(1)},\cdots,\ybf^{(k)}\}$ be a $\bbz_{\tilde{S}}$-basis for $\Delta$. 
%Similar to the
%argument of lemma~\ref{base}, one can find a $\bbz$-base
%$\{\ybf^{(1)},\cdots,\ybf^{(k)}\}$ for the group
%$$\widetilde{\Delta_f}=\Delta\cap\bbr^{m_{\infty}}\cdot\pf(a_{\nu}^{(0)}\bbz_{\nu}\oplus
%a_{\nu}^*\bbz_{\nu}^{d_{\nu}}\oplus
%a_{\nu}^{(1)}\bbz_{\nu}\oplus\cdots\oplus
%a_{\nu}^{(n)}\bbz_{\nu}).$$ Since, for any $\nu\in S_f$, $\u_{\x}^{\nu}$ %is in ${\rm
%GL}_{m_{\nu}}(\bbz_{\nu})$,
%$\{\dbf\u_{\x}\ybf^{(1)},\cdots,\dbf\u_{\x}\ybf^{(k)}\}$ is a good
%base of $\dbf\u_{\x}\Delta$. 
We have
$\theta(\dbf\Delta)=\theta(\dbf\mathcal{Y})$, where
$\mathcal{Y}=\ybf^{(1)}\wedge\cdots\wedge\ybf^{(k)}$. Let
$a_{\nu},b_{\nu}\in\bbq_{\nu}$ be such that
$$(\dbf\mathcal{Y})_{\nu}=a_{\nu} e_{\nu}^0\wedge
x_{\nu}^{(1)}\wedge\cdots\wedge x_{\nu}^{(k-1)}+b_{\nu}x_{\nu}^{(0)}\wedge\cdots\wedge x_{\nu}^{(k-1)}.$$ 
If $g(x)=(g_1(x),g_2(x))$ for $g_1$ and $g_2$ two functions from an open subset of
$\bbq_{\nu}^{d_{\nu}}$ to $\bbq_{\nu}$ define
$\widetilde{\g}^*(g)(x)=g_1(x)\g^*g_2(x)-g_2(x)\g^*g_1(x)$ where $\g^*\bar
g(x_{\nu})=\sum_{i=1}^{d_{\nu}}\d_i \bar g(x_{\nu})e_{\nu}^{*i}.$ 

%and $g_1$ and $g_2$ are two functions from an open subset of
%$\bbq_{\nu}^{d_{\nu}}$ to $\bbq_{\nu}$, and
%$g(x)=(g_1(x),g_2(x))$.
%Let $\g^*\bar
%g(x_{\nu})=\sum_{i=1}^{d_{\nu}}\d_i \bar g(x_{\nu})e_{\nu}^{*i},$
%for any function $\bar g$ from an open subset of
%$\bbq_{\nu}^{d_{\nu}}$ to $\bbq_{\nu}$, and define
%$\widetilde{\g}^*(g)(x)=g_1(x)\g^*g_2(x)-g_2(x)\g^*g_1(x)$, where
%$g_1$ and $g_2$ are two functions from an open subset of
%$\bbq_{\nu}^{d_{\nu}}$ to $\bbq_{\nu}$, and
%$g(x)=(g_1(x),g_2(x))$. 
Let us also define
$\hat{\mathbf{f}}(\x)=(\hat{f}_{\nu}(x_{\nu}))_{\nu\in S},$ where
$$\hat{f}_{\nu}(x_{\nu})=(1,0_{d_{\nu}},
\frac{a_{\nu}^{(1)}}{a_{\nu}^{(0)}}f_{\nu}^{(1)}(x_{\nu}),\cdots,
\frac{a_{\nu}^{(n)}}{a_{\nu}^{(0)}}f_{\nu}^{(n)}(x_{\nu})).$$ 
Manipulation of the formulas gives 
$$(\dbf\u_{\x}\dbf^{-1})_{\nu}w=w+(\hat{f}_{\nu}(x_{\nu})\cdot
w)e_{\nu}^0+\frac{a_{\nu}^{(0)}}{a_{\nu}^*}\g^*(\hat{f}_{\nu}(x_{\nu})w),$$
whenever $w$ is in $W_{\nu}\oplus\bbq_{\nu}e_{\nu}^n$. Therefore
we have

$$\pi_{\nu}((H(\x)\mathcal{Y})_{\nu})
=(a_{\nu}+b_{\nu}\hat{f}_{\nu}(x_{\nu})x_{\nu}^{(0)})e_{\nu}^0\wedge
x_{\nu}^{(1)}\wedge\cdots\wedge
x_{\nu}^{(k-1)}+b_{\nu}x_{\nu}^{(0)}\wedge\cdots\wedge
x_{\nu}^{(k-1)}$$
$$\hspace*{10mm}+b_{\nu}\sum_{i=1}^{k-1}
\pm(\hat{f}_{\nu}(x_{\nu})x_{\nu}^{(i)})e_{\nu}^0\wedge\bigwedge_{s\neq
i}
x_{\nu}^{(s)}+b_{\nu}\frac{a_{\nu}^{(0)}}{a_{\nu}^*}\sum_{i=0}^{k-1}
\pm\g^*(\hat{f}_{\nu}(x_{\nu})x_{\nu}^{(i)})\wedge\bigwedge_{s\neq
i} x_{\nu}^{(s)}$$

\begin{equation}~\label{eq:norm}+\frac{a_{\nu}^{(0)}}{a_{\nu}^*}\sum_{i=1}^{k-1}\pm
\widetilde{\g}^*(\hat{f}_{\nu}(x_{\nu})x_{\nu}^{(i)},a_{\nu}+b_{\nu}\hat{f}_{\nu}(x_{\nu})x_{\nu}^{(0)})
\wedge e_{\nu}^0\wedge\bigwedge_{s\neq
0,i}x_{\nu}^{(s)}\hspace{7mm}\end{equation}

$$+b_{\nu}\frac{a_{\nu}^{(0)}}{a_{\nu}^*}\sum_{i,j=1, j>i}^{k-1}\pm
\widetilde{\g}^*(\hat{f}_{\nu}(x_{\nu})x_{\nu}^{(i)},\hat{f}_{\nu}(x_{\nu})x_{\nu}^{(j)})
\wedge e_{\nu}^0\wedge\bigwedge_{s\neq
i,j}x_{\nu}^{(s)}.\hspace{7mm}$$

\noindent
The orthogonality assumption gives that the norm of the above vector would
be the maximum of norms of each of its summands. Hence we need to show each summand is a good function. Note that there is nothing to prove in the case $\nu=\infty.$ If $\nu\in S$ however our choice of $\mathbf{V}$ and conditions on $\fbf$ guarantee that we may apply corollary~\ref{linear} and theorem~\ref{tartibat} hence each summand is $(C_{\nu},\alpha_{\nu})$-good as we wanted.

%Using the fact
%that maximum of a family of $(C_{\nu},\alpha_{\nu})$-good functions is again a
%$(C_{\nu},\alpha_{\nu})$-good function, it suffices to show that the norm of
%each of these summands is a $(C_{\nu},\alpha_{\nu})$-good function for a fixed
%$C_{\nu}$ and $\alpha_{\nu}$. By theorem~\ref{linear}, we find a neighborhood
%$V_{\nu}^1$ of $x_{\nu}$, $C_{\nu}^1$ and $\alpha_{\nu}^1>0$ such that the first
%two lines would be $(C_{\nu}^1,\alpha_{\nu}^1)$-good functions on
%$V_{\nu}^1$. Also, theorem~\ref{tartibat} provides us a
%neighborhood $V_{\nu}^2$ of $x_{\nu}$, $C_{\nu}^2$, and $\alpha_{\nu}^2>0$ so that the rest would be $(C_{\nu}^2,\alpha_{\nu}^2)$-good functions. 
%Hence corollary 2.3 of~\cite{KT} gives us the claim, note that all these make sense for $\nu=\infty$ and/or that the goodness follows trivially for the infinite place, as $f_{\infty}=0$.\vspace{2mm}

%Hence
%$V_{\nu}=V_{\nu}^1\cap V_{\nu}^2$, $C_{\nu}=\max\{C_{\nu}^1,C_{\nu}^2\}$,
%and $\alpha_{\nu}=\min\{\alpha_1,\alpha_2\}$ give us the claim.
\vspace{2mm}
\noindent\textit{\textbf{Proof of (II).}} First line in
equation~(\ref{eq:norm}), gives that
$$\theta(\dbf\u_{\x}\Delta)\ge
\prod_{\nu\in\tilde{S}}\max\{|a_{\nu}+b_{\nu}\hat{f_{\nu}}(x_{\nu})\cdot x_{\nu}^{(0)}|,|b_{\nu}|\}$$ 
Thus
$\theta(\dbf\u_{\x}\Delta)\le 1$ implies\vspace{1mm} that
$\tp\max\{|a_{\nu}|,|b_{\nu}|\}$ has an upper bound. Hence corollary 7.9 of \cite{KT} finishes the proof of (II).

\vspace{2mm}
\noindent\textit{\textbf{Proof of (III).}} Let
$\bbf\subseteq\mathbf{V}$ be a ball containing $\x.$ Define 
$$\rho_1=\inf\{|f_{\nu}(x_{\nu})\cdot
C_{\nu}+c_{\nu}^{0}|_{\nu}\hspace{1mm}|\hspace{1mm}\x\in\bbf, \nu\in S,
C_{\nu}\in\bbq_{\nu}^n,
\para C_{\nu}\para=1, c_{\nu}^{0}\in\bbq_{\nu}\},$$
$$\rho_2=\inf\{\sup_{\x\in\bbf}\para\g
f_{\nu}(x_{\nu})C_{\nu}\para\hspace{1mm} |\nu\in S,
C_{\nu}\in\bbq_{\nu}^n, \para C_{\nu}\para=1\},\hspace{3cm}$$
Further let  $M=\sup_{\x\in\bbf}\max\{\para\fbf(\x)\para_S,\para\g\fbf(\x)\para_S\}$ and $\rho_3$ be the constant obtained by theorem~\ref{tartibat}(a).

Assume first that $\rank_{\bbz_{\tilde{S}}}\Delta=1.$ Hence $\Delta$ can be represented by a
vector $\mathbf{w}=(w_{\nu})_{\nu\in S}$, with $w^{i}_{\nu}\in\mathbb{Z}_{\tilde{S}}$ for all $i$'s and any $\nu\in\tilde{S}$. Now
$$c(\dbf\u_{\x}\wbf)\ge\min_i\left|\frac{1}{a_{\infty}^{(i)}}\right|_{\infty}\h\p\left|\frac{w_{\nu}^{(0)}+\sum_{i=1}^n
f_{\nu}^{(i)}(x_{\nu})w_{\nu}^{(i)}}{a_{\nu}^{(0)}}\right|_{\nu}\ge\rho_1^{\kappa}.$$ 
The proof in this case is complete.

Hence we may assume $\rank_{\bbz_{\tilde{S}}} \Delta=k>1$. With the notations as in part (I) let $x_{\nu}^{(1)},\cdots,x_{\nu}^{(k-2)}$ be an orthonormal set in $W_{\nu}\cap \Delta_{\nu}.$ We extend this to an orthonormal set in $(W_{\nu}\oplus \bbq_{\nu}e_{\nu}^n)\cap\Delta_{\nu}$ by adding $x_{\nu}^{(k-1)}$. Now if necessary choose a vector $x_{\nu}^{(0)}$ such that $\{e_{\nu}^0,x_{\nu}^{(0)},x_{\nu}^{(1)},\cdots,x_{\nu}^{(k-1)}\}$ is an orthonormal basis for $\Delta_{\nu}+\bbq_{\nu}e_{\nu}^0$.
%As in part (I), let us denote the $\mathbb{Q}_{\nu}$ span of the projection to $\nu$ place of $\Delta$ by $\Delta_{\nu}$. Let $x_{\nu}^{(1)},\cdots,x_{\nu}^{(k-2)}$ be an orthonormal set in $W_{\nu}\cap \Delta_{\nu}.$ 
Let $\mathcal{Y}=\ybf^{(1)}\wedge\cdots\wedge\ybf^{(k)}$ be as before. Since $D_{\nu}$ leaves $W_{\nu}, W_{\nu}^*, \bbq_{\nu} e_{\nu}^0$, and $\bbq_{\nu}e_{\nu}^n$ invariant, one has
$$\theta(\dbf\u_{\x}\Delta)=\theta(\dbf\u_{\x}\mathcal{Y})=\tp\theta_{\nu} (D_{\nu}\u_{\x}^{\nu}\mathcal{Y}_{\nu})=\tp \|D_{\nu}\pi_{\nu}(\u_{\x}^{\nu}\mathcal{Y}_{\nu})\|_{\nu}.$$
Let $a_{\nu},b_{\nu}\in\bbq_{\nu}$ be so that
$$\mathcal{Y}_{\nu}=a_{\nu} e_{\nu}^0\wedge
x_{\nu}^{(1)}\wedge\cdots\wedge
x_{\nu}^{(k-1)}+b_{\nu}x_{\nu}^{(0)}\wedge\cdots\wedge
x_{\nu}^{(k-1)}.$$
Note that $\tp\{|a_{\nu}|_{\nu},|b_{\nu}|_{\nu}\}\ge 1$. Let $\check{\fbf}(\x)=(\cf_{\nu}(x_{\nu}))_{\nu\in\ts}$ where $$\cf_{\nu}(x_{\nu})=(1,0_{d_{\nu}},f_{\nu}^{(1)}(x_{\nu}),\cdots,f_{\nu}^{(n)}(x_{\nu})).$$ 
We have

$$\pi_{\nu}(\u^{\nu}_{\x}\mathcal{Y}_{\nu})
=(a_{\nu}+b_{\nu}\cf_{\nu}(x_{\nu})x_{\nu}^{(0)})e_{\nu}^0\wedge
x_{\nu}^{(1)}\wedge\cdots\wedge
x_{\nu}^{(k-1)}+b_{\nu}x_{\nu}^{(0)}\wedge\cdots\wedge
x_{\nu}^{(k-1)}$$
$$\hspace*{10mm}+b_{\nu}\sum_{i=1}^{k-1}
\pm(\cf_{\nu}(x_{\nu})x_{\nu}^{(i)})e_{\nu}^0\wedge\bigwedge_{s\neq
i}
x_{\nu}^{(s)}+b_{\nu}\sum_{i=0}^{k-1}
\pm\g^*(\cf_{\nu}(x_{\nu})x_{\nu}^{(i)})\wedge\bigwedge_{s\neq
i} x_{\nu}^{(s)}$$
$$+e_{\nu}^0\wedge \check{\mathcal{Y}}_{\nu}(x_{\nu}),\hspace{7cm}$$
$$\mbox{where}\hspace{2mm}\check{\mathcal{Y}}_{\nu}(x_{\nu})=\sum_{i=1}^{k-1}\pm
\widetilde{\g}^*(\cf_{\nu}(x_{\nu})x_{\nu}^{(i)},a_{\nu}+b_{\nu}\cf_{\nu}(x_{\nu})x_{\nu}^{(0)})
\wedge\bigwedge_{s\neq 0,i}x_{\nu}^{(s)}\hspace{7mm}$$
$$+b_{\nu}\sum_{i,j=1, j>i}^{k-1}\pm
\widetilde{\g}^*(\cf_{\nu}(x_{\nu})x_{\nu}^{(i)},\cf_{\nu}(x_{\nu})x_{\nu}^{(j)})
\wedge\bigwedge_{s\neq
i,j}x_{\nu}^{(s)}.\hspace{7mm}$$
\textit{Calim:} For all $\nu\in S$ one has
$$\sup\|e_{\nu}^n\wedge\check{\mathcal{Y}}_{\nu}(x_{\nu})\|_{\nu}\geq \rho_0\cdot\max\{|a_{\nu}|_{\nu},|b_{\nu}|_{\nu}\}.$$
\textit{Proof of the claim:} Let $\nu\in S$ we have
$$e_{\nu}^n\wedge\check{\mathcal{Y}}_{\nu}(x_{\nu})=\pm z_{\nu}^{(*)}(x_{\nu})\wedge e_{\nu}^n\wedge x_{\nu}^{(1)}\wedge x_{\nu}^{(2)}\cdots\wedge x_{\nu}^{(k-2)}+\hspace{1mm}\begin{array}{l}\mbox{\rm{other terms where one}}\\ \mbox{\rm{or two}}\hspace{1mm} x_{\nu}^{(i)}\hspace{1mm}\mbox{\rm{are missing,}}\end{array}$$ where

\vspace{.5mm}
$z_{\nu}^{(*)}(x_{\nu})={\widetilde{\nabla}}^{*}(\check{f}_{\nu}(x_{\nu})x_{\nu}^{k-1},\hspace{1mm}a_{\nu}+b_{\nu}\check{f}_{\nu}(x_{\nu})x_{\nu}^{(0)}) $ $$=b_{\nu}{\widetilde{\nabla}}^{*}(\check{f}_{\nu}(x_{\nu})x_{\nu}^{k-1},\hspace{1mm}\check{f}_{\nu}(x_{\nu})x_{\nu}^{(0)})-a_{\nu}{\nabla}^{*}(\check{f}_{\nu}(x_\nu)x_{\nu}^{(k-1)})$$ Using the first expression it follows that $\sup_{x_{\nu}\in B_{\nu}}\|z_{\nu}^{(*)}(x_{\nu})\|_{\nu}\geq \rho_3 \hspace{1mm}|b_{\nu}|_{\nu},$ and the second expression gives, $\hspace{1mm}\sup_{x_{\nu}\in B_{\nu}}\|z_{\nu}^{(*)}(x_{\nu})\|_{\nu}\geq \rho_2|a_{\nu}|_{\nu}-2M^2|b_{\nu}|_{\nu}.$  Thus there exists $\rho_0$
such that $$\max\{\rho_2|a_{\nu}|_{\nu}-2M^2|b_{\nu}|_{\nu},\rho_3 \hspace{1mm}|b_{\nu}|_{\nu}\}\geq \rho_0\cdot\max\{|a_{\nu}|_{\nu},|b_{\nu}|_{\nu}\}.$$ This shows the claim.

Let $\nu\in S$ be any place then $\|D_{\nu}(e_{\nu}^n\wedge\check{\mathcal{Y}}_{\nu}(x_{\nu}))\|_{\nu}\leq\|D_{\nu}\check{\mathcal{Y}}_{\nu}(x_{\nu})\|_{\nu}/|a_{\nu}^{(n)}|_{\nu}.$ Hence ${(a_{\nu}^{(*)}a_{\nu}^{(n-k+2)}\cdots a_{\nu}^{(n)})}^{-1}$  
is the eigenvalue with the smallest norm of $D_{\nu}$ on $W_{\nu}^{*}\wedge (\bigwedge^{k-1}(\bbq_{\nu} e_{\nu}^0\oplus W_{\nu}\oplus\bbq_{\nu}e_{\nu}^n)).$

\vspace{1mm}
Let $\mathcal{R}=\frac{\max\{|a_{\infty}|,|b_{\infty}|\}}{|a_{\infty}^{(0)}a_{\infty}^{(1)}\cdots a_{\infty}^{(n)}|_{\infty}}$ we have 
$$\hspace{1mm}\sup_{x\in\mathbf{B}} \theta(\mathbf{D}\mathcal{U}_{\mathbf{x}}\mathcal{Y})\ge\mathcal{R}\p\|D_{\nu}(e_{\nu}^0\wedge \check{\mathcal{Y}}_{\nu}(x_{\nu}))\|_{\nu}\ge\mathcal{R}\p\frac{|a_{\nu}^{(n)}|_{\nu}\|D_{\nu}(e_{\nu}^n\wedge\check{\mathcal{Y}}_{\nu}(x_{\nu}))\|_{\nu}}{|a_{\nu}^{(0)}|_{\nu}}$$
$$\ge\mathcal{R}\p\frac{|a_{\nu}^{(n)}|_{\nu}\|e_{\nu}^n\wedge\check{\mathcal{Y}}_{\nu}(x_{\nu})\|_{\nu}}{|a_{\nu}^{(0)}a_{\nu}^{(*)}a_{\nu}^{(n-k+3)}\cdots a_{\nu}^{(n)}|_{\nu}}\ge$$ $$\frac{\rho_0^{\kappa}\max\{|a|_{\infty},|b|_{\infty}\}}{|a_{\infty}^{(0)}a_{\infty}^{(1)}\cdots a_{\infty}^{(n)}|_{\infty}}\p\frac{\max\{|a|_{\nu},|b|_{\nu}\}}{|a_{\nu}^{(0)}a_{\nu}^{(*)}a_{\nu}^{(n-k+3)}\cdots a_{\nu}^{(n-1)}|_{\nu}}\geq \rho_0^{\kappa}.$$
This finishes the proof of part (III). 

\vspace{1mm}
As mentioned before now theorem~\ref{poset} completes the proof of theorem~\ref{unipotent}.

%as we wanted. Thus it suffices to show $(*).$ To that end for any place $\nu\in S$ select the term containing $x_{\nu}^{(1)}\wedge x_{\nu}^{(2)}\cdots\wedge x_{\nu}^{(k-2)},$ then one has $$e_{\nu}^n\wedge\check{\mathcal{Y}}_{\nu}(x_{\nu})=\pm z_{\nu}^{(*)}(x_{\nu})\wedge e_{\nu}^n\wedge x_{\nu}^{(1)}\wedge x_{\nu}^{(2)}\cdots\wedge x_{\nu}^{(k-2)}+\hspace{1mm}\begin{array}{l}\mbox{\rm{other terms where one}}\\ \mbox{\rm{or two}}\hspace{1mm} x_{\nu}^{(i)}\hspace{1mm}\mbox{\rm{are missing,}}\end{array}$$ where

%\noindent
%$z_{\nu}^{(*)}(x_{\nu})={\widetilde{\nabla}}^{*}(\check{f}_{\nu}(x_{\nu})x_{\nu}^{k-1},\hspace{1mm}a_{\nu}+b_{\nu}\check{f}_{\nu}(x_{\nu})x_{\nu}^{(0)}) $ $$=b_{\nu}{\widetilde{\nabla}}^{*}(\check{f}_{\nu}(x_{\nu})x_{\nu}^{k-1},\hspace{1mm}\check{f}_{\nu}(x_{\nu})x_{\nu}^{(0)})-a_{\nu}{\nabla}^{*}(\check{f}_{\nu}(x_\nu)x_{\nu}^{(k-1)})$$ Using the first expression it follows that $\sup_{x_{\nu}\in B_{\nu}}\|z_{\nu}^{(*)}(x_{\nu})\|_{\nu}\geq \rho_3 \hspace{1mm}|b_{\nu}|_{\nu},$ where the second expression gives, $\hspace{1mm}\sup_{x_{\nu}\in B_{\nu}}\|z_{\nu}^{(*)}(x_{\nu})\|_{\nu}\geq \rho_2|a_{\nu}|_{\nu}-2M^2|b_{\nu}|_{\nu}.$  It is easy to see that there exists $\rho_0$
%such that $$\max\{\rho_2|a_{\nu}|_{\nu}-2M^2|b_{\nu}|_{\nu},\rho_3 \hspace{1mm}|b_{\nu}|_{\nu}\}\geq \rho_0\cdot\max\{|a_{\nu}|_{\nu},|b_{\nu}|_{\nu}\}.$$ This finishes the proof.

%%%%%%%%%%%%%%%%%%%%%%%%%%%%%%%%%%%%%%%%%%%%%%%%%%%%%%%%%%%%%%%%%%%%%%% REGULAR SYSTEMS %%%%%%%%%%%%%%%%%%%%%%%%%%%%%%%%%%%%%%%%%%%%%%%%%%%%%%%%%%%%%%%%%%%%%%%%

\section{Regular systems~\label{secregular}}
In this section we will prove theorem~\ref{regular} and will state a general result about regular systems, theorem~\ref{approximation}. Trough out this section $\U$ and $\fbf$ will be as in the theorem~\ref{regular}. Let us first recall the definition of {\it regular system of resonant sets} this is a generalization of the concept of regular system of points of Baker and Schmidt for the real line. 

\begin{definition}~\label{rsystem}(cf.~\cite[Definition 3.1]{BBKM})
Let $\U$ be an open subset of $\bbq_{\nu}^d,$ $\mathcal{R}$ be a family of subsets of $\bbq_{\nu}^d,$ $N:\mathcal{R}\rightarrow\bbr_+$ be a function and let $s$ be a number satisfying $0\leq s<d.$ The triple $(\mathcal{R}, N, s)$ is called a {\it regular system} in $\U$ if there exists constants $K_1, K_2, K_3>0$ and a function $\lambda:\bbr_+\rightarrow\bbr_+$ with $\lim_{x\rightarrow\infty}\lambda(x)=+\infty$ such that for any ball $\mathbf{B}\subset\U$ and for any $T>T_0=T_0(\mathcal{R}, N, s, \mathbf{B})$ is a sufficiently large number, there exists 
$$R_1,\cdots,R_t\in\mathcal{R}\hspace{3mm}\mbox{with}\hspace{3mm}\lambda(T)\leq N(R_i)\leq T\hspace{3mm}\mbox{for}\hspace{3mm}i=1,\cdots,t$$
and disjoint balls
$$\bbf_1,\cdots,\bbf_t\hspace{3mm}\mbox{with}\hspace{3mm}2\bbf_i\subset\bbf\hspace{3mm}\mbox{for}\hspace{3mm}i=1,\cdots,t$$
such that 
$$\mbox{diam}(\bbf_i)=T^{-1}\hspace{3mm}\mbox{for}\hspace{3mm}i=1,\cdots,t$$
$$t\geq K_1|\bbf|T^d$$
and such that for any $\gamma\in\bbr$ with $0<\gamma<T^{-1}$ one has
$$K_2\gamma^{d-s}T^{-s}\leq|\bbf(R_i,\gamma)\cap\bbf_i|$$
$$|\bbf(R_i,\gamma)\cap2\bbf_i|\leq K_3\gamma^{d-s}T^{-s}$$
where $\bbf(R_i,\gamma)$ is the $\gamma$ neighborhood of $R_i.$

The elements of $\mathcal{R}$ will be called {\it resonant sets}.

\end{definition}

\noindent
The construction of the desired regular system, which in some sense is the main result of this section, will make essential use of the following.
 
\begin{thm}~\label{upperbound}
Let $\fbf:\U\rightarrow\bbq_{\nu}^n$ be a non degenerate map at $\x\in\U.$ Then there exists a sufficiently small ball $\mathbf{B}_0\subset\U$ centered at $\x_0$ and a constant $C_0>0$ such that for any ball $\mathbf{B}\subset\mathbf{B}_0$ and any $\delta>0,$ for all sufficiently large $Q,$ one has $$|\mathcal{A}_{\fbf}(\delta;\mathbf{B};Q)|\leq C_0\delta|\mathbf{B}|,$$ where $$\mathcal{A}_{\fbf}(\delta;\mathbf{B};Q)=\bigcup_{\tq\in\bbz^{n+1}:0<\|\tq\|_{\infty}\leq Q}\{\x\in\mathbf{B}|\h|\tq\cdot\fbf(\x)|<\delta Q^{-n-1}\}$$
 \end{thm}
 
 \begin{proof}
Let $\mathbf{B}_0$ be an open ball around $\x_0$ for which the assertions of theorems~\ref{>} and~\ref{<} hold. We will show $\mathbf{B}_0$ satisfies the conclusion of the theorem. For any $\mathbf{B}\subset\mathbf{B}_0$ let $$\mathcal{A}^1(\delta;\mathbf{B};Q;\epsilon)=\left\{\x\in\mathbf{B}|\h\exists\hspace{.5mm}\tilde{\q},\h\|\tilde{\q}\|_{\infty}<Q,\h\begin{array}{l}|\tilde{\q}\cdot\mathbf{f}(\x)|_{\nu}\h<\delta Q^{-n-1} \\  \|\q\g\mathbf{f}(\x)\|_{\nu}>|\tq|_{\infty}^{-\epsilon}\end{array}\right\}$$
and $$\mathcal{A}^2(\delta;\mathbf{B};Q;\epsilon)=\left\{\x\in\mathbf{B}|\hspace{1mm}\exists
\tilde{\q},\h \|\tq\|<Q:\begin{array}{l}|\tilde{\q}\cdot\mathbf{f}(\x)|<\delta Q^{-n-1}\\
\|\g f_{\nu}(x)q\|_{\nu}\leq|\tq|_{\infty}^{-\epsilon}\end{array}\right\}.$$ One obviously has $\mathcal{A}(\delta;\mathbf{B};Q)\subset\mathcal{A}^1(\delta;\mathbf{B};Q;\epsilon)\cup\mathcal{A}^2(\delta;\mathbf{B};Q;\epsilon).$ Now one applies the bounds from theorems~\ref{>} and~\ref{<} for $|\mathcal{A}^1(\delta;\mathbf{B};Q;\epsilon)|$ and $|\mathcal{A}^2(\delta;\mathbf{B};Q;\epsilon)|$ respectively. These give 
$$|\mathcal{A}^1(\delta;\mathbf{B};Q;\epsilon)|\leq C_1\delta|\bbf|\hspace{3mm}\mbox{and}\hspace{3mm}|\mathcal{A}^1(\delta;\mathbf{B};Q;\epsilon)|\leq C_2(\delta Q^{\epsilon})|^{\frac{\alpha}{n+1}}\bbf|$$
where $\alpha>0.$ Combining these, one get the desired bound for $\mathcal{A}(\delta;\mathbf{B};Q).$

\end{proof}
 
We are now ready to prove theorem~\ref{regular}.

\vspace{1mm}
\noindent
\textbf{\textit {Proof of the theorem~\ref{regular}.}} Thanks to the non-degeneracy assumption, replacing $\U$ with a smaller neighborhood, we may and will assume $\fbf_1(\x)=\x_1.$ Moreover, we can choose $\mathbf{B}_0$ such that theorem \ref{upperbound} holds. Therefore the aforementioned theorem will guarantee that for any $\mathbf{B}\subset\mathbf{B}_0,$ one has
\begin{equation}~\label{equregular}|\mathcal{G}(\mathbf{B};\hspace{.5mm}\delta;\hspace{.5mm} Q)|\geq\frac{1}{2}|\mathbf{B}|\end{equation} for large enough $Q,$ where $$\mathcal{G}(\mathbf{B};\hspace{.5mm}\delta;\hspace{.5mm} Q)=\frac{3}{4}\mathbf{B}\setminus\mathcal{A}_{\fbf}(\frac{3}{4}\mathbf{B};\hspace{.5mm}\delta;\hspace{.5mm}Q)$$ 

\noindent
Let $\x\in\mathcal{G}((\mathbf{B};\hspace{.5mm}\delta;\hspace{.5mm} Q))$, applying Dirichlet's principle argument one gets an absolute constant $C$ such that for sufficiently large $Q$ one can solve the following system of inequalities$$\begin{cases}|\q\cdot\fbf(\x)+q_0|_{\nu}<C\delta^2Q^{-n-1}\\|q_i|_{\infty}<\delta^{-1} Q\hspace{7mm}i=0, 1,\cdots, n\\|q_i|_{\nu}<\delta\hspace{1.5cm}i=2,\cdots, n\end{cases}$$
This thanks to the fact that $\x\in\mathcal{G}((\mathbf{B};\hspace{.5mm}\delta;\hspace{.5mm} Q))$ says $T=\delta^{-n-1}Q^{n+1}$ will satisfy $Q^{n+1}\leq N(R_{\q,q_0})\leq T.$
 \vspace{.5mm}

\noindent
\textbf{\textit {First claim:}} Let $(\q,q_0)$ satisfy the above system of inequalities.  Define the function $\bff(\x)=\q\cdot\fbf(\x)+q_0,$ then one has $|\partial_1\bff(\x)|_{\nu}>\frac{\delta}{2}.$ 

Assume the contrary so $|\partial_1\bff(\x)|_{\nu}\leq\frac{\delta}{2}.$ This assumption gives $|q_1|_{\nu}<\delta.$ Now since we have $|\q\cdot\fbf(\x)+q_0|_{\nu}<C\delta^2Q^{-n-1},$ if $Q$ is sufficiently large, we will have $|q_0|_{\nu}<\delta.$ This says that we can replace $(\q,q_0)$ by $(\q',q'_0)=\frac{1}{p_{\nu}^l}(\q,q_0)$ and have $$\begin{cases}|\q'\fbf(\x)+q'_0|_{\nu}<C\delta Q^{-n-1}\\ |q'_i|_{\infty}<Q\hspace{7mm}i=0, 1,\cdots,n\end{cases}$$ 
This however contradicts our assumption that $\x\in\mathcal{G}((\mathbf{B};\hspace{.5mm}\delta;\hspace{.5mm} Q)).$ Hence we have that $|\partial_1\bff(\x)|_{\nu}>\frac{\delta}{2}.$ The first claim is proved. 
\vspace{1mm}

\noindent
\textbf{\textit{Second claim:}} There exists $\z\in R_{\q,q_0}$ such that $|\z-\x|_{\nu}<2C\delta Q^{-n-1},$ for large enough $Q.$

\vspace{1mm}
Using uniform continuity and the ultrametric inequality we get that there exists $r_1>0$ such that if $\|\x-\ybf\|_{\nu}<r_1$ then $|\partial_1\bff(\ybf)|_{\nu}>\frac{\delta}{2}.$ As $\x\in\frac{3}{4}\mathbf{B}$ we have $\mathbf{B}(\x,\mbox{diam}\h\mathbf{B})\subset\mathbf{B}.$ Define $r_0=\min(r_1,\mbox{diam}\h\mathbf{B}),$ so we have $|\partial_1\bff(\ybf)|_{\nu}>\frac{\delta}{2}$ for all $\ybf\in\mathbf{B}(\x,r_0).$ 

\noindent
Now if $\x=(\x_1,\cdots,\x_d)$ and $|\theta|_{\nu}<r_0$ then $\x_{\theta}=(\x_1+\theta,\x_2,\cdots,\x_d)\in \mathbf{B}(\x,r_0).$ Let $g(\theta)=\bff(\x_{\theta}).$ Then 
$$|g(0)|_{\nu}=|\bff(\x)|_{\nu}<C\delta^{2}Q^{-n-1}\hspace{2mm}\mbox{and}\hspace{2mm}|g'(0)|_{\nu}=|\partial_1\bff(\x)|_{\nu}>\frac{\delta}{2}.$$ 
We now apply Newton's method and get; There exists $\theta_0$ such that $g(\theta_0)=0$ and $|\theta_0|_{\nu}<2C\delta Q^{-n-1}.$ So if $Q>\frac{2C}{\delta^{1/{n+1}}}$ then we have $\x_{\theta_0}\in \mathbf{B}(\x,r_0).$ \vspace{.5mm}Hence there is $\z\in R_{\q,q_0}$ with $|\z-\x|_{\nu}<2C\delta Q^{-n-1}.$

%Let $\nu\in S$ be an arbitrary place, for simplicity in the notations we will denote by $x=\x_{\nu},\h B(x,r_0)=\mathcal{B}_{\nu}(\x_{\nu},r_0)$ and $F(x)=\bff_{\nu}(\x_{\nu}).$ Now if $x=(x_1,\cdots,x_d)$ and $|\theta|_{\nu}<r_0$ then $x_{\theta}=(x_1+\theta,x_2,\cdots,x_d)\in B(x,r_0).$ Let $g(\theta)=F(x_{\theta}).$ Then $$|g(0)|_{\nu}=|F(x)|_{\nu}<C\delta^{2\kappa}Q^{-n-1}\hspace{2mm}\mbox{and}\hspace{2mm}|g'(0)|_{\nu}=|\partial_1F(x)|_{\nu}>\frac{\delta}{2}.$$ Now we apply Newton's method and we get, there exists $\theta_0$ such that $g(\theta_0)=0$ and $|\theta_0|_{\nu}<2C\delta Q^{-n-1}.$ So if $Q>\frac{2C}{\delta^{1/{n+1}}}$ then we have $x_{\theta_0}\in B(x,r_0).$ We can find such $(\theta_0)_{\nu}$ for any $\nu\in S,$ then define $\z=(x_{(\theta_0)_{\nu}})_{\nu\in S}.$ Then $\z\in R_{\q,q_0}$ and we have $|\z-\x|_S<2C\delta Q^{-n-1}.$ 
\vspace{1mm}
\noindent
\textbf{\textit{Third claim:}} There is a constant $K_2$ so that for any $0<\gamma<T^{-1}$ we have $$K_1\gamma T^{-(d-1)}\leq |\mathbf{B}(R_{\q,q_0},\gamma)\cap\mathbf{B}(\z,T^{-1}/2)|$$ 

If $d=1$ we are done by taking $K_1=1/2$ so we assume $d>1.$ Let $\z=(\z_1, \cdots, \z_d)$ and $\z'=(\z_2, \cdots, \z_d)$ where $\z$ is as in second claim above.  Now for any $\ybf'=(\ybf_2,\cdots,\ybf_d)\in\bbq_{\nu}^{d-1}$ such that $|\ybf'-\z'|_{\nu}<C_1T^{-1}$ let $\ybf=(y_1,\ybf')=(y_1, \ybf_2,\cdots, \ybf_d)$ where $y_1\in\bbq_{\nu}.$ If $|y_1-\z_1|_{\nu}\leq T^{-1}/4$ then $\ybf\in \mathbf{B}(\z\hspace{.5mm},T^{-1}/4).$ 

\vspace{1mm}
We now want to show that for any $\ybf'$ with $|\ybf'-\z'|_{\nu}<C_1T^{-1}$ one can find $y_1(\ybf')\in\bbq_{\nu}$ such that $\ybf=(y_1(\ybf'),\ybf')\in R_{\q,q_0}\cap\mathbf{B}(\z,T^{-1}/4).$ First note that $\mathbf{B}(\z,T^{-1})\subset\mathbf{B}(\x,r_0).$ So if $|y_1-\z_1|_{\nu}<T^{-1}/4$ then $(y_1,\ybf')\in \mathbf{B}(\x,r_0).$ This thanks to our previous observations gives $|\partial_1\bff(\ybf)|_{\nu}>\delta/2.$ Now the Mean value theorem gives $$\bff(\ybf)=\bff(\z)+\nabla\bff(\z)\cdot(\ybf-\z)+\sum_{i,j}\Phi_{ij}(\bff)(\ybf_i-\z_i)(\ybf_j-\z_j).$$ Comparing the maximum of the norms using $\bff(\z)=0$ and $|\ybf'-\z'|_{\nu}<C_1T^{-1}$ we get that if $|y_1-\z_1|_{\nu}<T^{-1}/4$ then $|\bff(\ybf)|_{\nu}<T^{-1}/4.$ Again Newton's method helps to find $y_1(\ybf')$ with $|y_1(\ybf')-\z_1|_{\nu}\leq T^{-1}/4$ such that $\bff(y_1(\ybf'),\ybf')=0.$

\noindent
For any $0<\gamma<T^{-1}$ define $$\mathcal{A}(\gamma)=\{(\ybf_1(\ybf')+\mathbf{\theta},\ybf')|\h \|\ybf'-\z'\|_{\nu}<C_1T^{-1},\h |\mathbf{\theta}|_{\nu}\leq\gamma/2\}$$
The above gives $\mathcal{A}(\gamma)\subset \mathbf{B}(R_{\q,q_0},\gamma)\cap\mathbf{B}(\z,T^{-1}/4).$ So an application of Fubini finishes the proof of the third claim.

\vspace{1mm}
The proof of the theorem now goes as in~\cite{BBKM}, we recall the steps here for the sake of completeness. Assume $Q$ is large enough so that theorem~\ref{upperbound} holds. Choose a collection
$$(\q_1,q_{0,1},\z_1),\cdots,(\q_t,q_{0,t},\z_t)\in(\bbz^n\setminus\{0\})\times\bbz\times\bbf\hspace{2mm}\mbox{with}\hspace{2mm}\z_i\in R_{\q_i,q_{0,i}}$$ 
such that
$$Q^{n+1}=T\delta^{n+1}\leq N(R_{\q_i,q_{0,i}})\leq T=\delta^{-n-1}Q^{n+1}\hspace{2mm}(1\leq i\leq t)$$
and such that for any $\gamma$ with $0<\gamma<T^{-1}$ we have
$$K_2\gamma T^{-(d-1)}\leq |\bbf(R_{\q_i,q_{0,i}},\gamma)\cap\bbf(\z_i,T^{-1}/2)|\hspace{2mm}(1\leq i\leq t)$$
$$ |\bbf(R_{\q_i,q_{0,i}},\gamma)\cap\bbf(\z_i,T^{-1})|\geq K_3\gamma T^{-(d-1)}\hspace{4mm}(1\leq i\leq t)$$
Now by our above discussion for any point $\x\in\mathcal{G}(\bbf;\delta;Q)$ there is a triple
$$(\q,q_{0},\z)\in(\bbz^n\setminus\{0\})\times\bbz\times\bbf\hspace{2mm}\mbox{with}\hspace{2mm}\z\in R_{\q,q_{0}}$$
which satisfies the above claims. Since $t$ was chosen to be maximal there is an index $i\in\{1,\cdots,t\}$ such that
$$\bbf(\z_i,T^{-1}/2)\cap\bbf(\z,T^{-1}/2)\neq\emptyset$$
As a result we have $\|\z-\z_i\|<T^{-1}/2.$ This together with the second claim above gives $\|\x-\z_i\|<C_2T^{-1}.$ Thus
$$\mathcal{G}(\bbf;\delta;Q)\subset\bigcup_{i=1}^{i=t}\h\bbf\h(\z_i,C_2T^{-1})$$ 
This inclusion plus~(\ref{equregular}) above give
$$|\bbf|/2\leq|\mathcal{G}(\bbf;\delta;Q)|\leq t\cdot|\bbf(\mathbf{0},C_2)|T^{-d}$$
Therefore $t\geq K_1|\bbf|T^d,$ where $K_1=|2\bbf(\mathbf{0},C_2)|.$

Now $R_i=R_{\q_i,q_{0,i}}$ and $\bbf_{i}=\bbf(\z_i,T^{-1}/2)$ serve as the resonant sets and the desired balls in the definition~\ref{rsystem}. This finishes the proof of theorem~\ref{regular}.

\vspace{2mm}
The following is a general result on regular systems which is theorem 4.1 in \cite{BBKM}. The proof in there is only given for $\bbr^d$ however the same proof works for $\bbq_{\nu}^d$ and we will not reproduce the proof here.

\begin{thm}~\label{approximation}(cf.~\cite[Theorem 4.1]{BBKM})

\noindent
Let $\U$ be an open subset of $\bbq^d,$ and let $(\mathcal{R}, N, S)$ be a regular system in $\U.$ Let $\widetilde{\Psi}:\bbr_+\rightarrow\bbr_+$ be a non-increasing function such that the sum $$\sum_{k=1}^{\infty}k^{d-s-1}\widetilde{\Psi}(k)^{d-s}$$
diverges. Then for almost all points $\x\in\U$ the inequality 
$$\mbox{\rm{dist}}(\x,R)<\widetilde{\Psi}(N(R))$$
has infinitely many solutions $R\in\mathcal{R}.$
\end{thm}

%%%%%%%%%%%%%%%%%%%%%%%%%%%%%%%%%%%%%%%%%%%%%%%%%%%%%%%%%%%%%%%%%     PROOF OF THE MAIN THEOREM   %%%%%%%%%%%%%%%%%%%%%%%%%%%%%%%%%%%%%%%%%%%%%%%%%%%%%%%%%%%%%%%%%%%%% 

\section{Proofs of the main theorems}~\label{secmain}
We finally come to the proofs of theorems~\ref{khintchine} and~\ref{divergent}.  
\subsection{Proof of the convergence part}
Take $\x_0\in \mathbf{U}$.  Choose a neighborhood $\mathbf{V}\subseteq
\mathbf{U}$ of $\x_0$ and a positive number $\alpha$, as in
theorem~\ref{<}, and pick a ball $\bbf=\p B_{\nu}\subseteq
\mathbf{V}$ containing $\mathbf{x}_0$ such that the ball with the
same center and triple the radius is contained in $\mathbf{U}$. We
will show that $\bbf\cap \mathcal{W}_{\mathbf{f},\Psi}$ has measure
zero. For any $\tq\in \bbz^{n+1}\setminus\{0\}$, let
$$A_{\tq}=\{(x_{\nu})_{\nu\in S}\in B |\hspace{1mm}
|(f_{\nu}(x_{\nu}))\cdot
\tq|_S<\Psi(\tq)\}.$$
We need to prove that the set of points $\mathbf{x}$ in $\bbf$ which belong to infinitely many
$A_{\tq}$ for ${\tq}\in\bbz^{n+1}\setminus\{0\}$ has measure zero. Now let
$$A_{\ge
\tq}=\{\mathbf{x}\in A_{\tq}|\hspace{1mm}\para\g
f_{\nu}({x}_{\nu})q\para_{\nu}\ge |\tq|_{\infty}^{-\epsilon}\}
 \hspace{2mm}\&$$
$$A_{<\tq}=A_{\tq}\setminus A_{\ge\tq}.$$
Furthermore for any $t\in\mathbb{N}$ let 
$$\begin{array}{ll}\bar{A}_{\ge t}=\bigcup_{\tq\in\bbz^{n+1}, 2^{t}\le
\|\tq\|_{\infty}<2^{t+1}}A_{\ge\tq}\hspace{4mm} \& & \hspace{3mm}
\bar{A}_{<{t}}=\bigcup_{\tq\in\bbz^{n+1}, 2^{t}\le
\|\tq\|_{\infty}<2^{t+1}}A_{<\tq}\end{array}$$

\noindent
Recall that $\Psi$ is decreasing. Hence using theorem~\ref{>}, with $$2^{t}\leq\|\tq\|_{\infty}\leq2^{t+1}\h \mbox{and} \hspace{1.5mm} \delta=2^{t(n+1)}\Psi(2^t,\cdots,2^t),$$
we see that $|\bar{A}_{\geq{t}}|\leq C2^{t(n+1)}\Psi(2^{t},\cdots,2^{t}).$ Now a use of  Borel-Cantelli lemma, gives that almost every $\x\in\B$ is in at most finitely many sets $A_{\ge\tq}.$

\vspace{1mm}
We will be done if we show that the sum of the
measures of $\bar{A}_{<{t}}$'s  is convergent.\vspace{.75mm}
\noindent
The conditions posed on $\Psi$ imply easily that
$\Psi(\tq)\le\|\tq\|_{\infty}^{-(n+1)} $\vspace{1mm}
for large enough $\para\tq\para_{\infty}$. So if $2^{t}\le \|\tq\|_{\infty}<2^{t+1}$
then $\Psi(\tq)\le 2^{-t(n+1)}$ for large enough $t$. Now for such ${t}$ we may write $\bar{A}_{{t}}=\cup_{\nu\in S}\bar{A}_{{t},\nu}$ where each $\bar{A}_{{t},\nu}$ is contained
in the set defined in~\ref{<}, with $\delta=2^{-t(n+1)},\hspace{1mm}T_i=2^{t+1},\hspace{1mm}K_{\nu}=2^{-\epsilon t}$ and $K_{\omega}=1$ for $\omega\in S\setminus\{\nu\}$. It is not hard to verify the inequalities in the hypothesis of theorem~\ref{<}. Moreover, one has
 $$\vare^{n+1}=\max\{\delta^{n+1},\delta^{\kappa} T_0\cdots T_n\p K_{\nu}\}=2^{-\epsilon t}\delta^{\kappa} T_0\cdots T_n=C'2^{-\epsilon t},$$
 for some universal constant $C.$  So by theorem~\ref{<} and the choice of $\mathbf{V}$ and $\bbf$ measure of
 $\bar A_{{t}}$ is at most
 $$C 2^{\frac{-\alpha\epsilon t}{\kappa(n+1)}}|\bbf|.$$
 Therefore the sum of measures of $\bar A_{{t}}$'s is
 finite, thus another use of Borel-Cantelli lemma completes the proof of theorem~\ref{khintchine}.
 
\subsection{Proof of the divergence part.} Replacing $\U$ by a smaller neighborhood, if needed, we assume that theorem~\ref{regular} holds for $\U.$ Define now the sequence  $$\widetilde{\Psi}(x)=x^{-n/(n+1)}\psi(x^{1/(n+1)})$$
As $\psi$ was non-increasing we get that $\widetilde{\Psi}$ is non-increasing as well, and we have
$$\sum_{k=1}^{\infty}k^{d-s-1}\widetilde{\Psi}(k)^{d-s}=\sum_{k=1}^{\infty}\widetilde{\Psi}(k)=\sum_{\ell=1}^{\infty}\sum_{(\ell-1)^{n+1}<k\leq\ell^{n+1}}\widetilde{\Psi}(k)\geq$$
$$\sum_{\ell=1}^{\infty}\sum_{(\ell-1)^{n+1}<k\leq\ell^{n+1}}\ell^{-n}\psi(\ell)\geq\sum_{\ell=1}^{\infty}\psi(\ell)=\infty$$
Now theorem~\ref{approximation} says, for almost every $\x\in\U$ there are infinitely many elements $(\q,q_0)\in\bbz^n\times\bbz$ satisfying 
\begin{equation}~\label{distance}\mbox{\rm{dist}}(\x,R_{\q,q_0})<\widetilde{\Psi}(\|\tq\|_{\infty}^{n+1})\end{equation}
this says, there is a point $\z\in R_{\q,q_0}$ such that 
$$\|\x-\z\|_{\nu}<\widetilde{\Psi}(\|\tq\|_{\infty}^{n+1}).$$
%Since $\bff(\z)=\q\cdot\fbf(\z)+q_0=0$ 
We now apply Mean value theorem
$$\bff(\x)=\bff(\z)+\nabla\bff(\x)\cdot(\x-\z)+\sum_{i}\Phi_{ij}(\bff)(t_i,t_j)\cdot(\x_i-\z_i)(\x_j-\z_j).$$
Recall that we have $\|\tq\|_{\nu}\leq1,$ $\|\nabla\fbf\|_{\nu}\leq1,$ $|\Phi_{i,j}(\fbf)|_{\nu}\leq1$ and $\bff(\z)=0.$ Hence we get
\begin{equation}~\label{tamam}|\tq\cdot\fbf(\x)|_{\nu}=|\bff(\x)|_{\nu}\leq\|\x-\z\|_{\nu}<\widetilde{\Psi}(\|\tq\|_{\infty}^{n+1})=\|\tq\|_{\infty}^{-n}\psi(\|\tq\|_{\infty})=\Psi(\tq)\end{equation}
Note that for almost every $\x\in\U$ there are infinitely many $(\q,q_0)\in\bbz^{n}\times\bbz,$ which satisfy ~(\ref{distance}). So we get that for almost every $\x\in\U$ there are infinitely many $(\q,q_0)\in\bbz^n\times\bbz$ satisfying $(\ref{tamam}).$ This completes the proof of theorem~\ref{divergent}.

 %%%%%%%%%%%%%%%%%%%%%%%%%%%%%%%%%%%%%%%%%%%%%%%%%%%%%%%%%%%%%%%%%%%%%%%%%%%%%%%%%%%%%%%%%%%%%%%%%%%%%%%%%%%%%%%%%%%%Remark and Open Problem%%%%%%%%%%%%%%%%%%%%%%%%%%%%%%%%%%%%%%%%%%%%%%%%%%%%%%%%%%%%%%%%%%%%%%%%
\section{A few remarks and open problems}\label{secremarks}
\textbf{1.} In this article, we worked with product of non-degenerate $p$-adic analytic manifolds. However most of the argument is valid for the product of non-degenerate $C^k$ manifolds. The only part in which we use analyticity extensively is in the proof of theorem~\ref{tartibat}. 

\textbf{2.} The divergence result in here was obtained for the case $S$ is a singleton only. However the convergence part of this paper which gives the simultaneous approximation should be optimal and the divergence should hold in the more general setting of the simultaneous approximation as well. 

\textbf{3.} We considered measured supported on non-degenrate manifolds in here. There are other natural measures that one can consider. Indeed D.~Y.~Kleinbock, E.~Lindenstrauss, B.~Weiss in~\cite{KLW} proved extremality of non-planer fractal measures. It is interesting to prove a Khintchine type theorem for fractal measures.

{\sc Dept. of Math., Univ. of Chicago, Chicago, IL, 60615}

{\em E-mail address:} {\tt amirmo@math.uchicago.edu}

{\sc Dept. of Math, Princeton Univ., Princeton, NJ, 08544}

{\em E-mail address:} {\tt asalehi@Math.Princeton.EDU}

\end{document}